\documentclass{article}
\usepackage{fullpage}
\usepackage{amsmath, amssymb}
\usepackage{mathtools}
\usepackage{amsthm} 
\usepackage{enumerate}  
\usepackage{url} 
\usepackage{multirow} 
\usepackage{caption}
\usepackage{subcaption}
\usepackage{tikz}
\usetikzlibrary{decorations.markings}
\pgfdeclarelayer{background}
\pgfdeclarelayer{nodelayer}
\pgfdeclarelayer{edgelayer}
\pgfsetlayers{background,edgelayer,nodelayer,main}

\tikzstyle{new}=[circle,  minimum width=4pt,inner sep=0pt, fill=black,draw=black]
\tikzstyle{none}=[circle,fill=white,draw=black]
\tikzstyle{n}=[shape=rectangle,minimum width=1pt,inner sep=0pt, fill=none,draw=none]
\tikzstyle{emph}=[circle,  minimum width=4pt,inner sep=0pt, fill=magenta,draw=magenta]

\tikzset{directed/.style={decoration={
  markings,
  mark=at position .6 with {\arrow{>}}},postaction={decorate}}}
\tikzset{identify1/.style={decoration={
  markings,
  mark=at position .7 with {\arrow{>}}},postaction={decorate}}}
\tikzset{identify2/.style={decoration={
  markings,
  mark=at position .6875 with {\arrow{>}},
  mark=at position .7125 with {\arrow{>}}},postaction={decorate}}}

\usepackage{array}
\usepackage{float}
\usepackage{wrapfig}

\usepackage{graphicx}
\usepackage{rotating}
\usepackage{nicefrac}
\usepackage{longtable}
\usepackage{xcolor}

\usepackage[overload]{textcase} 

\makeatletter
\def\cleardoublepage{
  \clearpage
  \if@twoside\ifodd\c@page\else
  \hbox{}
  \thispagestyle{empty}
  \newpage
  \if@twocolumn\hbox{}\newpage\fi
  \fi\fi
  }
\makeatother
\usepackage{hyperref}

\newtheoremstyle{plainsl}%
	{\topsep}
	{\topsep}
	{\slshape} % only non-default setting
	{}
	{\normalfont\bfseries}
	{.}
	{ }
	{}

\swapnumbers

{\theoremstyle{plainsl}
\newtheorem{theorem}{Theorem}[section]
\newtheorem{lemma}[theorem]{Lemma}
\newtheorem{proposition}[theorem]{Proposition}
\newtheorem{corollary}[theorem]{Corollary}
\newtheorem{conjecture}[theorem]{Conjecture}
\newtheorem{openprob}[theorem]{Open Problem}
}
{\theoremstyle{remark}
}

\numberwithin{equation}{section}

\newcommand\revarc[1]{{\mkern5mu\overleftarrow{\mkern1mu#1}}}
\newcommand\bmat[1]{\begin{bmatrix} #1 \end{bmatrix}}
\DeclareMathOperator\rk{rk}

\DeclareMathOperator\tr{tr}

\DeclareMathOperator\Aut{Aut}

\DeclareMathOperator\vecspan{span}
\DeclareMathOperator{\col}{col}

\newcommand\cx{{\mathbb C}}% complexes

\newcommand\ints{{\mathbb Z}}
\newcommand\rats{{\mathbb Q}}
\newcommand\cA{{\mathcal A}}

\newcommand\cF{{\mathcal F}}

\newcommand\cJ{{\mathcal J}}
\newcommand\cK{{\mathcal K}}

\newcommand\cW{{\mathcal W}}

\newcommand\Ze{{\mathbf e}}

\newcommand\Zv{{\mathbf v}}
\newcommand\Zw{{\mathbf w}}

\newcommand\Zx{{\mathbf x}}

\newcommand\ones{{\mathbf 1}}

\newcommand\ket[1]{| #1 \rangle}
\newcommand\inprod[2]{\langle#1,#2 \rangle}

\newcommand\Mhat{\widehat{M}}
\newcommand\Nhat{\widehat{N}}
\newcommand\Chat{\widehat{C}}

\allowdisplaybreaks
\definecolor{vcolour}{RGB}{230,97,0}
\definecolor{kcolour}{RGB}{93,58,155}

\newcommand{\arxiv}[1]{\href{https://arxiv.org/abs/#1}{\texttt{arXiv:#1}}}

\newcommand\krystalsays[1]{{\bf \textcolor{kcolour}{} }}
\newcommand\vincentsays[1]{{\bf \textcolor{vcolour}{} }}

\title{Perfect state transfer in quantum walks \\on orientable maps}
\author{Krystal Guo\thanks{Korteweg-de Vries Institute for Mathematics, University of Amsterdam, Amsterdam, The Netherlands.  QuSoft (Research center for Quantum software \& technology), Amsterdam, The Netherlands. \texttt{k.guo@uva.nl}} \and Vincent Schmeits\thanks{Korteweg-de Vries Institute for Mathematics, University of Amsterdam, Amsterdam, The Netherlands. \texttt{v.f.schmeits@uva.nl}}}
\date{November 23, 2022}

\begin{document}

\maketitle
\vspace{-20pt}
\begin{abstract}
 A discrete-time quantum walk is the quantum analogue of a Markov chain on a graph. Zhan [{\em 	J. Algebraic Combin.} 53(4):1187--1213, 2020] proposes a model of discrete-time quantum walk whose transition matrix is given by two reflections, using the face and vertex incidence relations of a graph embedded in an orientable surface. We show that the evolution of a  general discrete-time quantum walk that consists of two reflections satisfies a Chebyshev recurrence, under a projection. For the vertex-face walk, we prove theorems about perfect state transfer and periodicity and give infinite families of examples where these occur. 
 We bring together tools from algebraic and topological graph theory to analyze the evolution of this walk. 

    \vspace{5pt}
  \noindent\textit{Keywords: quantum walk, graph embeddings, graph eigenvalues} 
 
  \noindent\textit{Mathematics Subject Classifications 2020: 05C50, 05C10, 81P45  } 
  %05C50 Graphs and linear algebra (matrices, eigenvalues, etc.) 
  %05E18 Group actions on combinatorial structures
  %05C10 (1973-now) Planar graphs; geometric and topological aspects of graph theory
  %81P45 (2010-now) Quantum information, communication, networks (quantum-theoretic aspects) [See also 94A15, 94A17]
\end{abstract}

\vspace{5pt}
\centerline{
\begin{minipage}{0.7\linewidth}
\small
\tableofcontents
\end{minipage}}

\section{Introduction}
% General intro to topic and paper 
Quantum computing gives rise to many interesting applications of combinatorics; in this paper, we bring together ideas from algebraic and topological graph theory to study properties, including state transfer and periodicity, of a model of quantum walk which takes place on an embedded graph. 
Like their continuous-time counterpart, discrete-time quantum walks are computational primitives; 
\cite{LovCooEve2010} show that the discrete-time quantum walk is able to implement the same universal gate set and thus any quantum algorithm can be viewed as a discrete-time quantum walk. See
\cite{Por2013} for connections between quantum walks and quantum search. 
In recent papers and an upcoming book, Godsil and Zhan \cite{GodZha2019, Zha2020} describe the various models of discrete-time quantum walks and apply techniques from algebraic graph theory to study properties of the evolution of these walks.
In this paper, we prove a result about a general model of discrete walks, where the transition matrix consists of two reflections, and results about perfect state transfer and periodicity in one specific model, the vertex-face model. 

% introduction to quantum walks 
The quantum walks studied here are discrete-time and they are built from two reflections; the transition matrix $U$ is of the form
\[
U = (2P - I)(2Q- I),
\]
where $P,Q$ are the orthogonal projectors onto two subspaces, as defined  in \cite{MagNayRol2011}. These walks are referred to as  \textsl{bipartite walks} in \cite{CheGodSub2022} and are a general model of quantum walk which encompasses the walks defined by  Szegedy in his seminal paper \cite{Sze2004} and also includes the \textsl{vertex-face walks} which are the focus of this paper. 
For background on the role of discrete-time quantum walks in quantum algorithms, we refer to \cite{San2008, ApeGilJef2019}. We will defer the definition of the vertex-face walks until Section \ref{sec:vxfacewalkdefn}; intuitively, the walk evolves on a graph embedded in an orientable surface and the transition matrix has the property that $P$ and $Q$ are the projections onto vector spaces determined by incidence relations of faces and of vertices,  respectively. The vertex-face model was first defined in \cite{Zha2020}, motivated by spatial quantum search in  \cite{PatRagRun2005,Falk2013,AmbPorNah2015}, where the quantum walk used corresponds, in some way, to the vertex-face walk on the toroidal grids.

%%%%%%%%%%

\begin{figure}[ht]
\centering
\begin{subfigure}{0.21\textwidth}
    \centering
    \includegraphics[width=\textwidth]{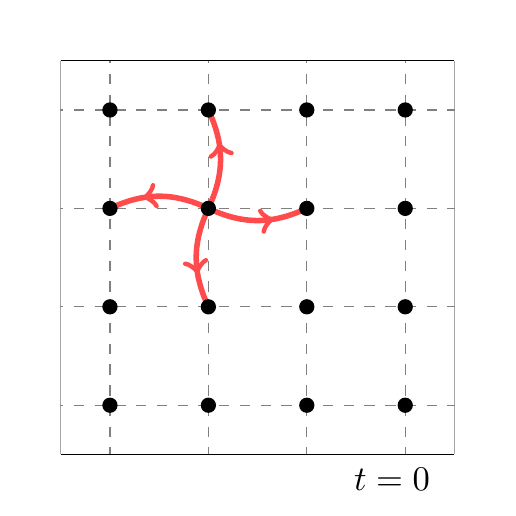}
\end{subfigure}
\kern-1em
\begin{subfigure}{0.21\textwidth}
    \centering
    \includegraphics[width=\textwidth]{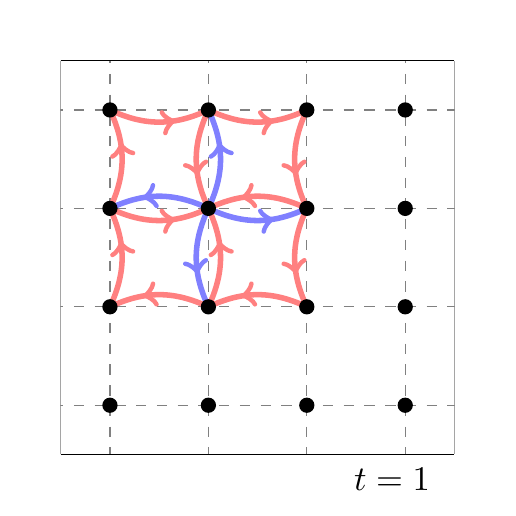}
\end{subfigure}
\kern-1em
\begin{subfigure}{0.21\textwidth}
    \centering
    \includegraphics[width=\textwidth]{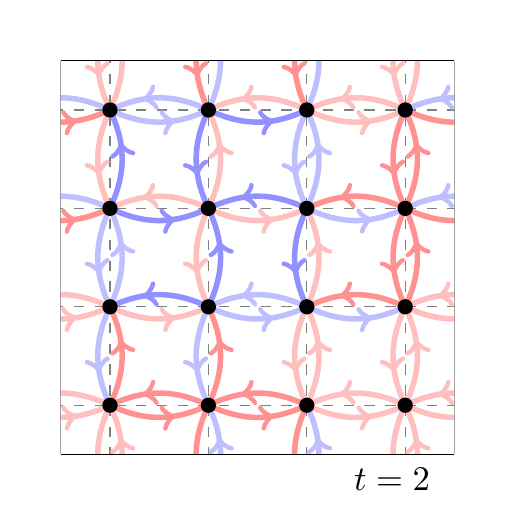}
\end{subfigure}
\kern-1em
\begin{subfigure}{0.21\textwidth}
    \centering
    \includegraphics[width=\textwidth]{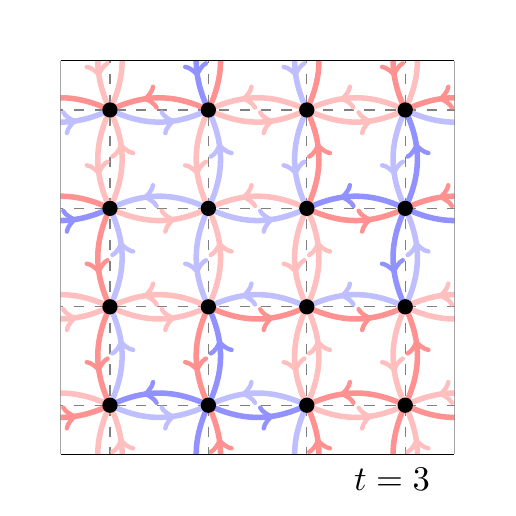}
\end{subfigure}
\kern-1em
\begin{subfigure}{0.21\textwidth}
    \centering
    \includegraphics[width=\textwidth]{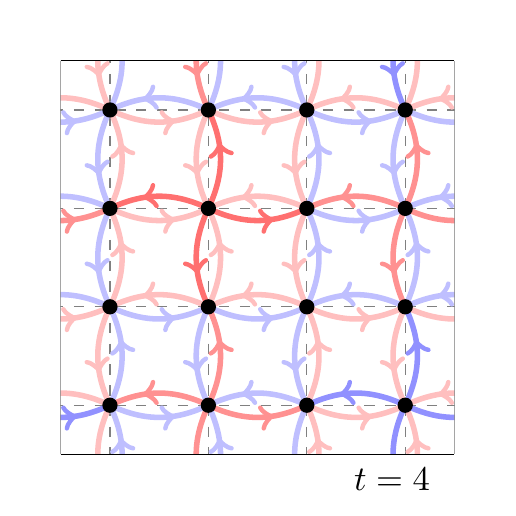}
\end{subfigure}

\vspace{-4pt}

\begin{subfigure}{0.21\textwidth}
    \centering
    \includegraphics[width=\textwidth]{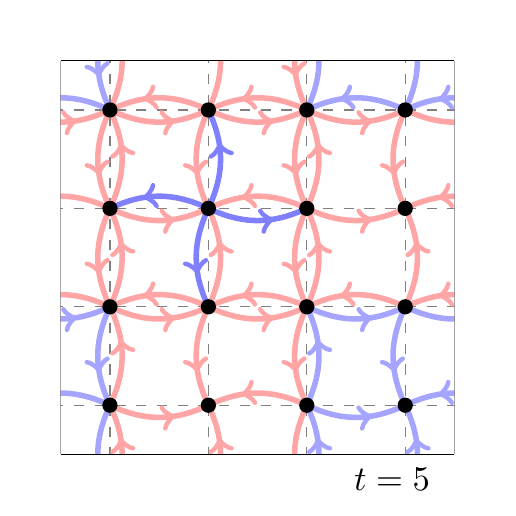}
\end{subfigure}
\kern-1em
\begin{subfigure}{0.21\textwidth}
    \centering
    \includegraphics[width=\textwidth]{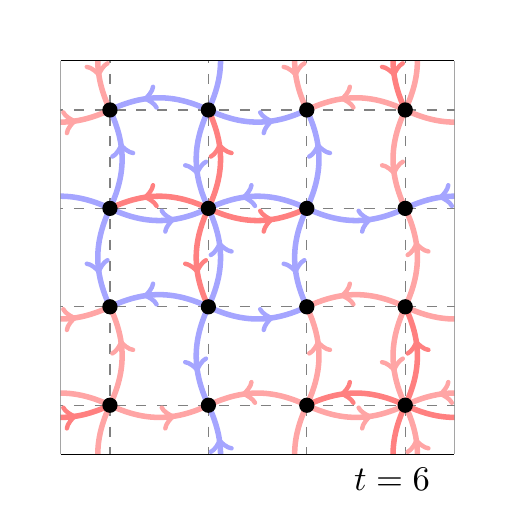}
\end{subfigure}
\kern-1em
\begin{subfigure}{0.21\textwidth}
    \centering
    \includegraphics[width=\textwidth]{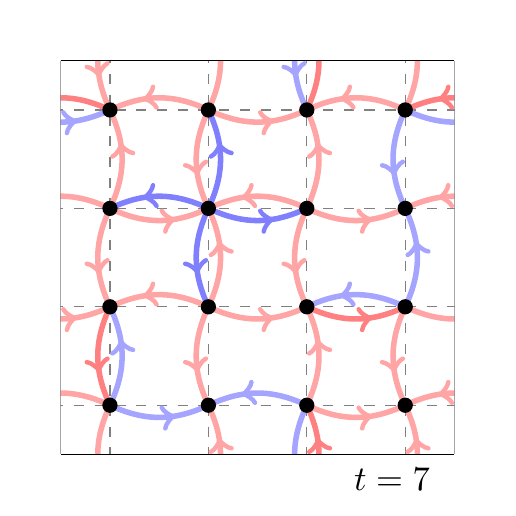}
\end{subfigure}
\kern-1em
\begin{subfigure}{0.21\textwidth}
    \centering
    \includegraphics[width=\textwidth]{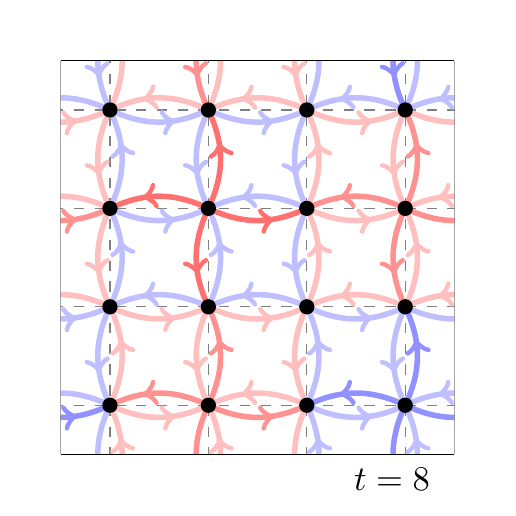}
\end{subfigure}
\kern-1em
\begin{subfigure}{0.21\textwidth}
    \centering
    \includegraphics[width=\textwidth]{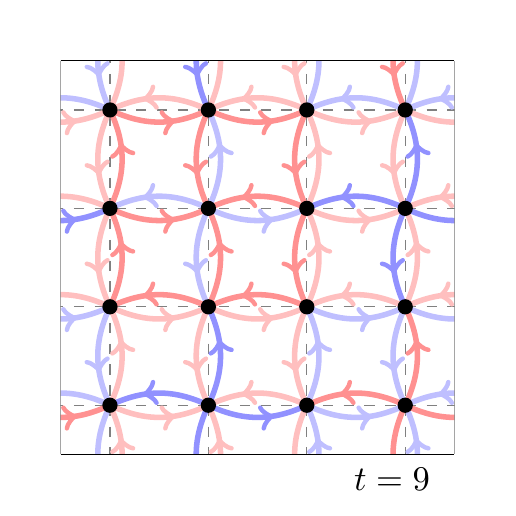}
\end{subfigure}

\vspace{-4pt}

\begin{subfigure}{0.21\textwidth}
    \centering
    \includegraphics[width=\textwidth]{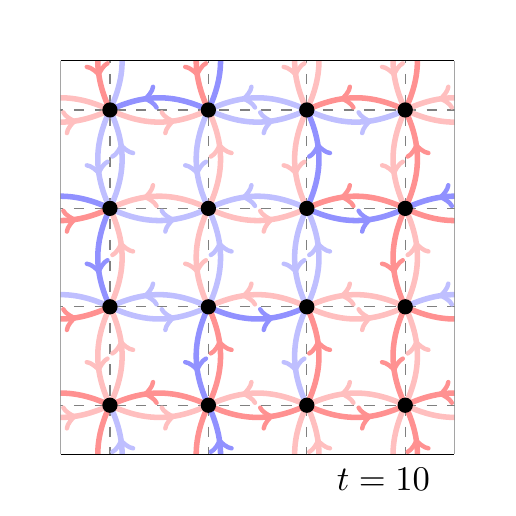}
\end{subfigure}
\kern-1em
\begin{subfigure}{0.21\textwidth}
    \centering
    \includegraphics[width=\textwidth]{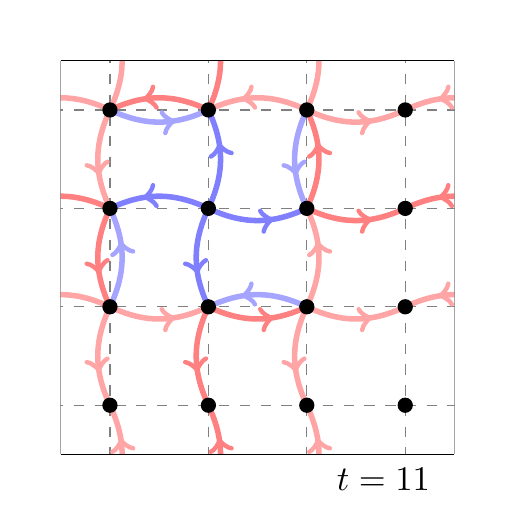}
\end{subfigure}
\kern-1em
\begin{subfigure}{0.21\textwidth}
    \centering
    \includegraphics[width=\textwidth]{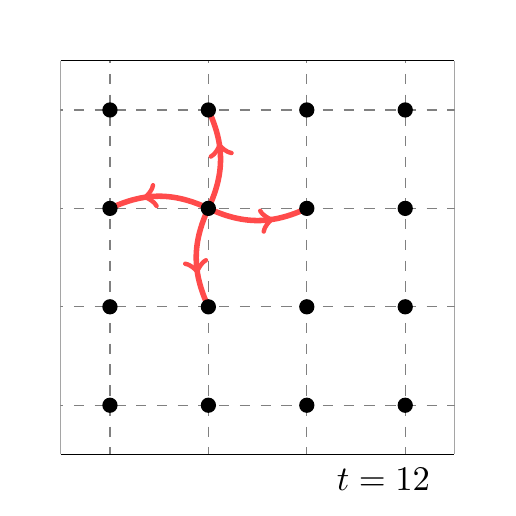}
\end{subfigure}
\kern-1em
\begin{subfigure}{0.21\textwidth}
    \centering
    \includegraphics[width=\textwidth]{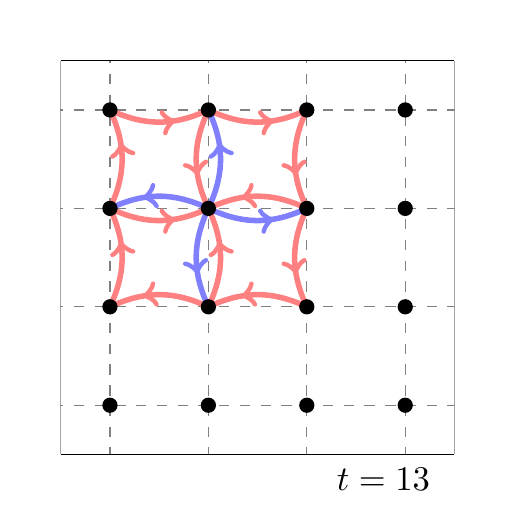}
\end{subfigure}
\kern-1em
\begin{subfigure}{0.21\textwidth}
    \centering
    \includegraphics[width=\textwidth]{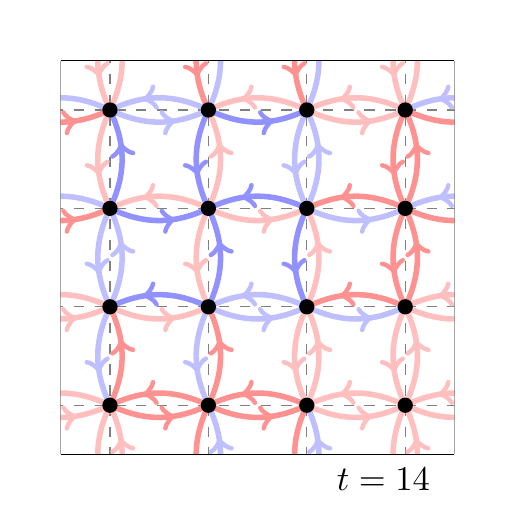}
\end{subfigure}

\vspace{-4pt}
\caption{The vertex-face walk on the toroidal $(4,4)$-grid. Each diagram consists of the $4 \times 4$ grid embedded in the torus, shown here as a cut-open torus, where the opposite sides are identified; for visual simplicity, we have omitted the labels on the boundary of the torus.  The state space of the walk is the space of arcs. Here we represent the amplitude of each arc by using opacity for magnitude and colours (red, blue) for the sign (positive, negative, resp.).\label{fig:4_4grid}}
\end{figure}
%%%%%%%%%

% intro to state transfer 
Figure \ref{fig:4_4grid} shows the evolution of the vertex-face walk on the $4 \times 4$ toroidal grid. We will now give an intuitive was\vincentsays{?} the quantum walk properties studied herein, and defer rigourous definitions until Section \ref{sec:pst}. In the Figure, we see that the state at times $t=0$ and $t=12$ are identical; this is called \textsl{periodicity}. If the state had moved to the same distribution, but at another vertex, it would be an example of \textsl{perfect state transfer}. 
Though state transfer and periodicity has been studied in continuous-time quantum walks in a combinatorial setting, see \cite{GodKriSev2012, CoutinhoGodsilGuoVanhove2, GodGuoKem2020} for examples,  it is a relatively unexplored topic for bipartite walks and for the vertex-face walks in particular, though some recent papers have appeared; for example, pretty good  state transfer in discrete quantum walks has been studied in \cite{ChaZha2021}. 
The relationship between continuous and discrete quantum walks is explicated in \cite{Chi2010}. 
We make our own contribution by establishing some fundamental properties of state transfer in the vertex-face model of discrete-time quantum walks, with some analoguous theorems to those for continuous quantum walks. 

% Outline main results
The main results of this paper are as follows. 
First we consider, the general model of discrete-time quantum walks and we give a surprising Chebyshev recurrence for its evolution with respect to one of the reflections, in Theorem \ref{thm:B_Chebyshev}. 
Applying this recurrence to our the vertex-face walk, we establish fundamental properties of perfect state transfer in Theorem \ref{thm:pst_reverse_periodicity}. 
We show that, loosely speaking, if the map admits perfect state transfer everywhere,  then it also admits periodicity and has the property that there is some $\tau > 0$ such that $U^{\tau} = I$. 
We then characterise maps for which $U^{\tau} = I$ fully for $\tau=1,2$, and give partial results for larger $\tau$. We give new examples of perfect state transfer in infinite families of maps (dipoles and grids). 

% organization of the paper. 
The organization of this paper is as follows.
Since the vertex-face walk takes place on a cellularly embedded graph and is defined with incidence matrices which are not standard in the literature, we give the necessary preliminaries on graph embeddings in Section \ref{sec:preliminaries}.
In Section \ref{sec:vxfacewalkdefn}, we give the formal definition of the vertex-face quantum walk. 
In Section \ref{sec:pst}, we prove a general result about a Chebyshev recurrence for discrete-time quantum walk, and apply it in the specific model, the vertex-face walk, to prove our main results on perfect state transfer.
We move to more symmetric graphs in Section \ref{sec:periodicity} and establish the connection between perfect state transfer, periodicity and maps where some power of the transition matrix equals the identity. 
In Section \ref{sec:powerofUisI}, we work towards characterisation for those maps where some power of the transition matrix is the identity matrix. 
We give three infinite families of examples of perfect state transfer in Section \ref{sec:infexamples}.  Since the vertex-face walk is a relatively new concept and not many examples are well-understood, we performed computations pertaining to our main results on the census of regular maps, as given by Conder in \cite{Con2009}, in Section \ref{sec:computations} to gain intuition on these walks. 
Finally, we conclude with open problems in Section \ref{sec:conclusion}. 

\section{Preliminaries}
\label{sec:preliminaries}

Before we can give the formal definition of a vertex-face walk, we have to refresh our definitions and notation for graph embeddings. The vertex-face walk is defined for a graph embedded in an orientable surface, using the incidence relations between its vertices, faces and edges. 
We consider graphs with loops and parallel edges; we will use `graph' and `multigraph' interchangeably. 
At the end of the section, we turn our attention to automorphisms of maps and define (orientably-)regular maps, which provide a broad class of examples that can be searched computationally in order to gain intuition on vertex-face walks. 

We consider cellular embeddings of graphs on orientable surfaces. A \textsl{map} is a $2$-cell embedding of a connected graph into a closed surface with no boundary. A map is completely determined by its facial boundary walks. A map is \textsl{orientable} if the underlying surface is orientable. In this paper, we will exclusively consider orientable maps and will often write ``map'' for ``orientable map'' for convenience. The number of handles is the called the \textsl{genus} of the surface. The \textsl{genus} $g$ of an orientable map is equal to the genus of its underlying surface and satisfies \textsl{Euler's formula}:
\[
|V| - |E| + |F| = 2 - 2g,
\]
where $V$, $E$ and $F$ are respectively the sets of vertices, edges and faces of the map.
For background on maps, surfaces and topological graph theory, we refer the reader to \cite{GroTuc1987} and \cite{MohTho2001}.

On an orientable surface, we can make a consistent distinction between a `clockwise' and a `anticlockwise' orientation. For each vertex of an orientable map, we can give a cyclic ordering of the edges and faces incident to that vertex, using the clockwise order in which these edges and faces are attached to that vertex. For example, for $v_1$ of the embedded digon in Figure \ref{fig:digon&dual}, it is $(e_1,f_1,e_2,f_2)$. The subsequence of edges is said to be the \textsl{rotation} of the vertex and the set of all rotations form the \textsl{rotation system} of the map. Every orientable map is, up to homeomorphism, uniquely defined by its rotation system. With the clockwise orientation, the edges incident to a face can be ordered similarly; the \textsl{facial walk} is the alternating sequence of incident vertices and edges in the clockwise order in which they appear on the boundary of that face. For the  example shown in Figure \ref{fig:digon&dual}, the facial walk of $f_1$ of $X_2$ is given by $(v_1, e_1,v_2,e_2)$.

The \textsl{dual} $X^*$ of $X$ is the map whose vertex set is the set of faces of $X$, whose edge set is equal to that of $X$, where the rotational system is given by the facial boundary walks of $X$. Note that the dual is also a 2-cell embedding in the same surface. Figure \ref{fig:digon&dual} depicts the digon (left) and its dual (right) embedded in the sphere (genus $0$). We will denote this map by $X_2$. We have $V = \{v_1,v_2\}$, $E = \{e_1,e_2\}$ and $F = \{f_1,f_2\}$. Note that $X_2$ is \textsl{self-dual}: there exist bijections $V \to F$ and $E \to E$ that preserve the incidence structure of the map. (In particular, the graphs underlying the map and its dual are isomorphic.)

\begin{figure}[htbp]
    \centering
    \begin{tikzpicture}
    	\begin{pgfonlayer}{nodelayer}
    		\node [label=below:$v_1$] (0) at (-1.5, 0) {};
    		\node [label=right:$v_2$] (1) at (1.5, 0) {};
    		\node [label=left:$f_1$] (2) at (0, 0) {};
		    \node [label=below:$f_2$] (3) at (3, 0) {};
		    \node [label=center:$X_2$] at (-1,-1) {};
    	\end{pgfonlayer}
    	\begin{pgfonlayer}{edgelayer}
    		\draw [thick, color=black, bend left=50] (0.center) to (1.center);
    		\draw [thick, color=black, bend left=50] (1.center) to (0.center);
    		\draw [dotted, color=gray, bend left=50] (2.center) to node [pos=0.275, above, text=black] {$e_1$} (3.center);
    		\draw [dotted, color=gray, bend left=50] (3.center) to node [pos=0.725, below, text=black] {$e_2$} (2.center);
    	\end{pgfonlayer}
    	
       \filldraw [black]
    	(0) circle (2pt)
      	(1) circle (2pt);
       \filldraw [gray]
        (2) circle (1pt)
        (3) circle (1pt);
    \end{tikzpicture}
    \quad\quad\quad
    \begin{tikzpicture}
    	\begin{pgfonlayer}{nodelayer}
    		\node [label=below:$v_1$] (0) at (-1.5, 0) {};
    		\node [label=right:$v_2$] (1) at (1.5, 0) {};
    		\node [label=left:$f_1$] (2) at (0, 0) {};
		    \node [label=below:$f_2$] (3) at (3, 0) {};
		    \node [label=center:$X_2^*$] at (2.5,-1) {};
    	\end{pgfonlayer}
    	\begin{pgfonlayer}{edgelayer}
    		\draw [dotted, color=gray, bend left=50] (0.center) to (1.center);
    		\draw [dotted, color=gray, bend left=50] (1.center) to (0.center);
    		\draw [thick, color=black, bend left=50] (2.center) to node [pos=0.275, above, text=black] {$e_1$} (3.center);
    		\draw [thick, color=black, bend left=50] (3.center) to node [pos=0.725, below, text=black] {$e_2$} (2.center);
    	\end{pgfonlayer}
    	
       \filldraw [gray]
    	(0) circle (1pt)
      	(1) circle (1pt);
       \filldraw [black]
        (2) circle (2pt)
        (3) circle (2pt);
    \end{tikzpicture}
    \caption{$X_2$ and its dual $X_2^*$ embedded in the sphere.\label{fig:digon&dual}}
\end{figure}
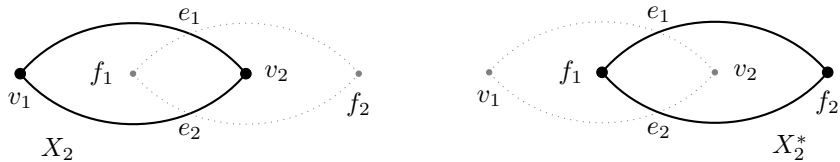

The \textsl{(vertex-)degree} of a vertex is the number of edges in its rotation. Note that each loop contributes $2$ to the degree of the vertex that it is attached to. Likewise, the \textsl{(face-)degree} of a face is the number of edges in its facial walk. If an edge appears twice in the facial walk (which implies that it is a loop in the dual), then it contributes $2$ to the face-degree. A \textsl{type $(k,d)$} map is a map where every vertex has degree $d$  and every face has degree $k$. The dual of a type $(k,d)$ map is a type $(d,k)$ map. The map $X_2$ is a type $(2,2)$ map.

Let $X$ be a map and assume, for now, that both $X$ and its dual have no loops. A \textsl{flag} of $X$ is defined as a triple $(v,e,f)$ of a vertex $v$, an edge $e$ and a face $f$ of $X$ that are all pairwise incident to each other. The set of all flags is denoted by $\cF$ and we have $|\cF| = 4|E|$, as every edge is incident to four distinct flags. For example, in Figure \ref{fig:4flags}, the green triangle represents the flag $(u,e,f)$. In the example in Figure \ref{fig:digon&dual}, the flags are formed by all possible triples:
\[
\begin{array}{cccc}
(v_1, e_1, f_1), &(v_1, e_2, f_1), &(v_1, e_2, f_2), &(v_1, e_1, f_2), \\
(v_2, e_2, f_1), &(v_2, e_1, f_1), &(v_2, e_1, f_2), &(v_2, e_2,f_2).
\end{array}
\]

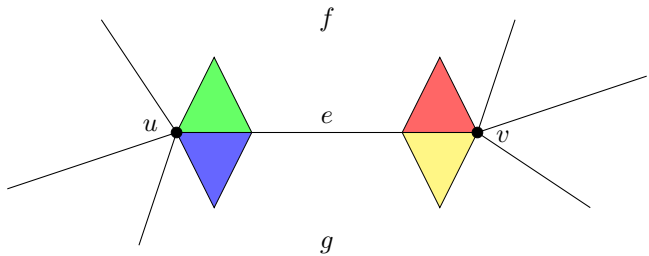
\begin{figure}
    \centering
    \begin{tikzpicture}
    	\begin{pgfonlayer}{nodelayer}
    		\node [label={[shift={(0,0.1)}]left:$u$}] (0) at (-2, 0) {};
    		\node [label={[shift={(0,-0.05)}]right:$v$}] (1) at (2, 0) {};
    		\node (2) at (-3, 1.5) {};
    		\node (3) at (-2.5, -1.5) {};
    		\node (4) at (-4.25, -0.75) {};
    		\node (5) at (2.5, 1.5) {};
    		\node (7) at (4.25, 0.75) {};
    		\node (8) at (3.5, -1) {};
    		\node (11) at (-1, 0) {};
    		\node (14) at (1, 0) {};
    		\node (15) at (-1.5, 1) {};
    		\node (16) at (-1.5, -1) {};
    		\node (17) at (1.5, 1) {};
    		\node (18) at (1.5, -1) {};
    		\node [label=center:$f$] (19) at (0, 1.5) {};
    		\node [label=center:$g$] (20) at (0, -1.5) {};
    	\end{pgfonlayer}
    	\begin{pgfonlayer}{edgelayer}
    		\draw (0.center) to node [above, text=black] {$e$} (1.center);
    		\draw (2.center) to (0.center);
    		\draw (0.center) to (4.center);
    		\draw (0.center) to (3.center);
    		\draw (1.center) to (5.center);
    		\draw (1.center) to (7.center);
    		\draw (1.center) to (8.center);
    	\end{pgfonlayer}
    	\begin{pgfonlayer}{background}
    		\fill[fill=blue!60, draw=black, ultra thin] (0.center) to (16.center) to (11.center) to (0.center);
    		\fill[fill=green!60, draw=black, ultra thin] (0.center) to (15.center) to (11.center) to (0.center);
    		\fill[fill=red!60, draw=black, ultra thin] (1.center) to (17.center) to (14.center) to (1.center);
    		\fill[fill=yellow!60, draw=black, ultra thin] (1.center) to (18.center) to (14.center) to (1.center);
    	\end{pgfonlayer}
       \filldraw [black]
       	(0) circle (2pt)
        (1) circle (2pt);
    \end{tikzpicture}

    \caption{The coloured triangles represent the flags that are incident to the edge $e$. The blue and red flags are the clockwise flags.\label{fig:4flags}}
\end{figure}

Since $X$ is orientable, we can make a distinction between flags that are clockwise and flags that are anticlockwise, by the direction in which the flag `points'. In Figure \ref{fig:4flags}, the blue and red flags are the clockwise flags, and the green and yellow flags are oriented anticlockwise. The clockwise flags in Figure \ref{fig:digon&dual} are
\[
\begin{array}{cccc}
(v_1, e_1, f_1), &(v_1, e_2, f_2), &(v_2, e_2, f_1), &(v_2, e_1, f_2).
\end{array}
\] 

Given an orientable map $X$, we can, for every non-loop edge, add a pair of arcs pointing in opposite directions, and positioned on opposite sides of that edge. As such, each arc lies inside a face of $X$. Because of the orientability of $X$, this can be done in such a way that each arc is pointed in the direction of the facial walk of its corresponding face. We denote the set of arcs by $\cA$. In Figure \ref{fig:digon_arcs}, the arcs of $X_2$ are depicted.

\begin{figure}[H]
    \centering
    \begin{tikzpicture}[scale=1.5]
    	\begin{pgfonlayer}{nodelayer}
    		\node [label=left:$v_1$] (0) at (-1.5, 0) {};
    		\node [label=right:$v_2$] (1) at (1.5, 0) {};
            \node [label={[shift={(0,-.15)}]right: $f_1$}] (2) at (0, 0) {};
		    \node [label=below left:$f_2$] (3) at (3, 0) {};
    	\end{pgfonlayer}
    	\begin{pgfonlayer}{edgelayer}
    		\draw [style=directed, thick, color=black, bend left=35] (0.center) to node [pos = 0.6, below] {$a_2$} (1.center);
    		\draw [style=directed, thick, color=black, bend left=35] (1.center) to node [pos = 0.6, above] {$a_3$} (0.center);
    		\draw [style=directed, thick, color=black, bend right=70] (1.center) to node [pos = 0.6, above] {$a_1$} (0.center);
    		\draw [style=directed, thick, color=black, bend right=70] (0.center) to node [pos = 0.6, below] {$a_4$} (1.center);
    		\draw [thin, dashed, color=gray, bend left=52.2] (0.center) to node [below, shift={(0,.075)}] {$e_1$} (1.center);
    		\draw [thin, dashed, color=gray, bend left=52.2] (1.center) to node [above, shift={(0,-.1)}] {$e_2$} (0.center);
    	\end{pgfonlayer}
        \filldraw [black]
            (0) circle (2pt)
            (1) circle (2pt);
    \end{tikzpicture}
    \caption{The arcs of $X_2$.\label{fig:digon_arcs}}
\end{figure}
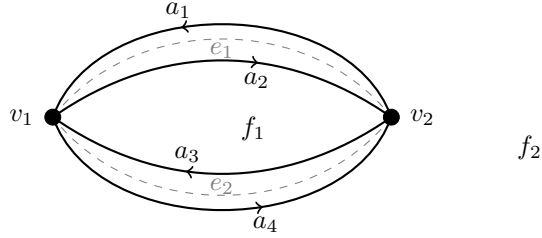

For our initial definition of a flag, we assumed that both $X$ and its dual have no loops. In that case, there is a clear 1-1 correspondence between the arcs of $X$ and its clockwise flags. For example, in Figure \ref{fig:digon_arcs}, the arc $a_1$ corresponds to the flag $(v_2,e_1,f_2)$: the tail of $a_1$ is $v_2$, and the arc lies inside $f_2$, alongside the edge $e_1$. If $X$ or its dual $X^*$ has a loop however, we require a more abstract definition which allows for multiple flags to be incident to the same vertex, edge and face. We define the set of flags $\cF$ to be an abstract set with an incidence function $\phi : \cF \to V\times E \times F$, such that every edge $e$ is incident to four unique flags. Intuitively, we would like the four flags shown in Figure \ref{fig:4flags} to be distinct objects. 
Figure \ref{fig:2loop_torus} shows a graph consisting of a single vertex with two loops attached, embedded as an orientable map on the torus. It has a single vertex $v$ and a single face $f$, so $|V \times E \times F| = 2$. The flags $\mathbf{f}_1$, $\mathbf{f}_4$, $\mathbf{f}_5$ and $\mathbf{f}_8$ are incident to the triple $(v,e_1,f)$ and the flags $\mathbf{f}_2$, $\mathbf{f}_3$, $\mathbf{f}_5$ and $\mathbf{f}_6$ are incident to $(v,e_2,f)$. The clockwise flags are $\mathbf{f}_1$, $\mathbf{f}_3$, $\mathbf{f}_5$ and $\mathbf{f}_7$, and they correspond to the arcs of the map.

\begin{figure}[H]
    \centering
    \begin{tikzpicture}[scale=1.3]
    	\begin{pgfonlayer}{nodelayer}
    		\node (1) at (-3, -2) {};
    		\node (3) at (3, 2) {};
    		\node (4) at (3, -2) {};
    		\node (5) at (-3, 2) {};
    		\node (6) at (0, 2) {};
    		\node (7) at (0, -2) {};
    		\node (8) at (-3, 0) {};
    		\node (9) at (3, 0) {};
    		\node [label={[shift={(-.11,-.28)}]center:{$v$}}] (10) at (0, 0) {};
    		\node (11) at (-.75,-.75) {};
    		\node (12) at (-.75,.75) {};
    		\node (13) at (.75,-.75) {};
    		\node (14) at (.75,.75) {};

    		\node [label={[shift={(.05,-.05)}]center:{$\mathbf{f}_1$}}] (a1) at (.1875,.5625) {};
    		\node [label={[shift={(-.025,0)}]center:{$\mathbf{f}_2$}}] (a2) at (.5625,.1875) {};
    		\node [label={[shift={(-.025,.-.025)}]center:{$\mathbf{f}_3$}}] (a3) at (.5625,.-.1875) {};
    		\node [label={[shift={(.05,0)}]center:{$\mathbf{f}_4$}}] (a4) at (.1875,-.5625) {};
    		\node [label={[shift={(-.025,0)}]center:{$\mathbf{f}_5$}}] (a5) at (-.1875,-.5625) {};
    		\node [label={[shift={(.05,-.025)}]center:{$\mathbf{f}_6$}}] (a6) at (-.5625,-.1875) {};
    		\node [label={[shift={(.05,0)}]center:{$\mathbf{f}_7$}}] (a7) at (-.5625,.1875) {};
    		\node [label={[shift={(-.05,-.05)}]center:{$\mathbf{f}_8$}}] (a8) at (-.1875,.5625) {};
    		
    		\node [label=center:$f$] (f1) at (-1.5,1) {};
    		\node [label=center:$f$] (f2) at (1.5,1) {};
    		\node [label=center:$f$] (f3) at (-1.5,-1) {};
    		\node [label=center:$f$] (f4) at (1.5,-1) {};
    	\end{pgfonlayer}
    	\begin{pgfonlayer}{background}
    	    \filldraw [draw=black, fill=lightgray]
    	        (11.center) to (12.center) to (14.center) to (13.center) to (11.center);
    	    \draw (11.center) to (14.center);
    	    \draw (12.center) to (13.center);
    	\end{pgfonlayer}
    	\begin{pgfonlayer}{edgelayer}
    		\draw (10.center) to node [pos=.66, right] {$e_1$} (6.center);
    		\draw (10.center) to node [pos=.66, below] {$e_2$} (9.center);
    		\draw (10.center) to node [pos=.66, left] {$e_1$} (7.center);
    		\draw (10.center) to node [pos=.66, above] {$e_2$} (8.center);
    		\draw [thin, style=identify2](1.center) to (5.center);
    		\draw [thin, style=identify2](4.center) to (3.center);
    		\draw [thin, style=identify1] (5.center) to (3.center);
    		\draw [thin, style=identify1] (1.center) to (4.center);
    	\end{pgfonlayer}
        \filldraw [black]
            (10) circle (2pt);
    \end{tikzpicture}
    \caption{One vertex with two loops embedded in the torus, shown as a cut-open torus. It has $8$ flags.\label{fig:2loop_torus}}
\end{figure}
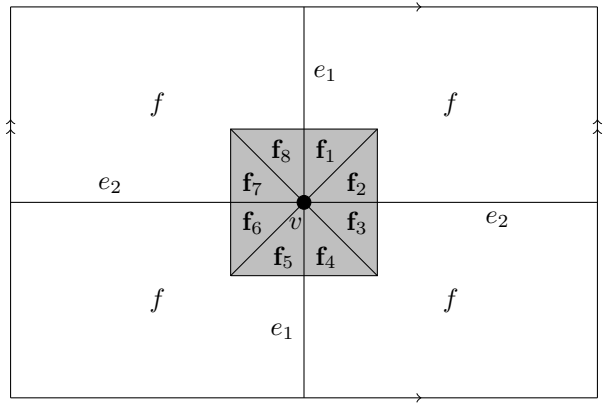

Given an orientable map $X$, we need to define several incidence matrices; the state space for the quantum  walk that we study is the set of arcs of the graph. For an arc $a \in \cA$, let $v(a)$ be the tail vertex of $a$, let $f(a)$ be the face in which $a$ lies, and let $e(a)$ be the edge along which $a$ lies. (Alternatively, $(v(a),e(a),f(a))$ is the incident triple for the corresponding clockwise flag.) We define the \textsl{arc-vertex incidence matrix} $N \in \{0,1\}^{\cA \times V}$, the \textsl{arc-face incidence matrix} $M \in \{0,1\}^{\cA \times F}$ and the \textsl{arc-edge incidence matrix} $L \in \{0,1\}^{\cA \times E}$ as follows:
\[
N(a,v) = \begin{cases}
1 \quad &\text{if $v = v(a)$;} \\
0 \quad &\text{otherwise,}
\end{cases} \quad 
M(a,f) = \begin{cases}
1 \quad &\text{if $f = f(a)$;} \\
0 \quad &\text{otherwise,}
\end{cases}
\quad \text{and} \quad 
L(a,e) = \begin{cases}
1 \quad &\text{if $e = e(a)$;} \\
0 \quad &\text{otherwise.}
\end{cases}
\]

The incidence matrices for $X^*$ (with the same set of arcs $\cA$) are obtained by reversing the roles of $N$ and $M$. Note that these are different incidence matrices from those often considered in the literature; these incidence matrices capture incidence relations on arcs, instead of incidence relations on edges. For the map $X_2$, the matrices can be written down explicitly as follows:
\[
N = \kern-.4em
\begin{array}{cc}
  & \begin{array}{cc}v_1 \!\!& \!\! v_2 \end{array} \\
\begin{array}{c} a_1 \kern-1.5em \\ a_2  \kern-1.5em\\ a_3 \kern-1.5em \\ a_4 \kern-1.5em \end{array} & \left[\begin{array}{cc}0 & 1 \\ 1 & 0 \\ 0 & 1 \\ 1 & 0\end{array}\right]
\end{array}, \quad\quad
M = \kern-.4em
\begin{array}{cc}
  & \begin{array}{cc}f_1 \!\!& \!\! f_2 \end{array} \\
\begin{array}{c} a_1 \kern-1.5em \\ a_2  \kern-1.5em\\ a_3 \kern-1.5em \\ a_4 \kern-1.5em \end{array} & \left[\begin{array}{cc}0 & 1 \\ 1 & 0 \\ 1 & 0 \\ 0 & 1\end{array}\right]
\end{array},
\quad \text{and} \quad
L = \kern-.4em
\begin{array}{cc}
  & \begin{array}{cc}e_1 \!\!& \!\! e_2 \end{array} \\
\begin{array}{c} a_1 \kern-1.5em \\ a_2  \kern-1.5em\\ a_3 \kern-1.5em \\ a_4 \kern-1.5em \end{array} & \left[\begin{array}{cc}1 & 0 \\ 1 & 0 \\ 0 & 1 \\ 0 & 1\end{array}\right]
\end{array}.
\]

For any arc $a \in A$, we denote by $\revarc{a}$ for the unique other arc incident to the edge $e(a)$; that is, $\revarc{a}$ is the arc going in the opposite direction from $a$. For instance, in Figure \ref{fig:digon_arcs}, $\revarc{a_1} = a_2$. The \textsl{arc-reversal matrix} $R \in \{0,1\}^{\cA \times \cA}$ is the permutation matrix that switches each such pair of arcs: it is defined by
\[
R(a,b) = \begin{cases}
1 \quad &\text{if $\revarc{a} = b$}; \\
0 &\text{otherwise.}
\end{cases}
\]
Alternatively, we can write $R = LL^T - I$. In the case of our example $X_2$, the matrix $R$ is given by
\[
R = I \otimes \bmat{0 & 1 \\ 1 & 0}.
\]

Let $D \in \cx^{V \times V}$ be the diagonal matrix for which the diagonal $(v,v)$-entry is equal to the degree of the vertex $v$. Similarly, let $\Delta \in \cx^{F \times F}$ be the diagonal matrix for which the diagonal $(f,f)$-entry is equal to the degree of the vertex $f$. We have the following lemma:

\begin{proposition}
\label{prop:NLM_props}
The following properties hold:
\begin{enumerate}[(i)]
\item $N^TN = D$, $L^TL = 2I$ and $M^TM = \Delta$;
\item $N^TRN = A(X)$ and $M^TRM = A(X^*)$;
\item $N^TRM = N^TM$.
\end{enumerate}
Here, $A(X)$ denotes the adjacency matrix of the graph underlying the map $X$.
\end{proposition}
\begin{proof} The proofs of parts (i) and (ii) are relatively straightforward. Here, we will only prove that $N^TRN = A(X)$.
We consider the entries of $N^TRN$: let $u,v \in V$, then 
\[
\begin{split}
(N^TRN)(u,v) &= \sum_{a,b\in\cA} N(a,u)R(a,b)N(b,v) \\
&= |\{(a,b) \in \cA^2 : a \neq b, e(a) = e(b), v(a) = u, v(b) = v\}| \\
&= |\{ e \in E: \text{$u$ and $v$ are endpoints of $e$}\}|.
\end{split}
\]
Thus $N^TRN = A(X)$. The remaining identities of (i) and (ii) can be proved in a similar fashion.

For (iii), let $v \in V$ and $f \in F$. We have
\[
\begin{split}
(N^TRM)(v,f) &= \sum_{a,b} N(a,v)R(a,b)M(b,f) \\
&= |\{a \in \cA : v(a) = v, f(\revarc{a}) = f\}| \\
&= |\{b \in \cA : v(b) = v, f(b) = f\}| \\
&= (N^TM)(v,f).
\end{split}
\]
Here, the third equality follows from the following observation: if $a$ is an arc that is incident to $v$ and such that $\revarc{a}$ is incident to $f$, then the arc $b$ that precedes $a$ in the rotation of $v$ is incident to both $v$ and $f$; see Figure \ref{fig:NTRM=NRM}. Hence there is a bijection between the two sets before and after the third equality.
\end{proof}

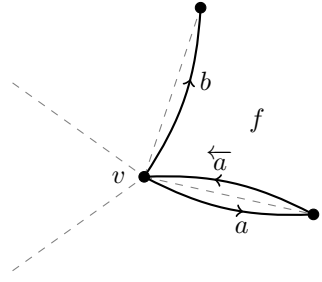
\begin{figure}
    \centering
    \begin{tikzpicture}
    	\begin{pgfonlayer}{nodelayer}
    		\node [label=left:$v$] (0) at (0, 0) {};
    		\node (1) at (0.75, 2.25) {};
    		\node (2) at (-1.75, 1.25) {};
    		\node (3) at (-1.75, -1.25) {};
    		\node (5) at (2.25, -0.5) {};
    		\node [label=center:$f$] (6) at (1.5, .75) {};
    	\end{pgfonlayer}
    	\begin{pgfonlayer}{edgelayer}
    		\draw [gray, dashed] (0.center) to (1.center);
    		\draw [gray, dashed] (0.center) to (5.center);
    		\draw [gray, dashed] (3.center) to (0.center);
    		\draw [gray, dashed] (2.center) to (0.center);
    		\draw [style=directed, thick, color=black, bend right=15] (0.center) to node [pos = 0.6, below] {$a$} (5.center);
    		\draw [style=directed, thick, color=black, bend right=15] (5.center) to node [pos = 0.6, above] {$\revarc{a}$} (0.center);
    		\draw [style=directed, thick, color=black, bend right=15] (0.center) to node [pos = 0.6, right] {$b$} (1.center);
    	\end{pgfonlayer}
    	
    	\filldraw [black]
    	    (0) circle (2pt)
    	    (1) circle (2pt)
    	    (5) circle (2pt);
    \end{tikzpicture}

    \caption{If $v(a) = v$ and $f(\protect\revarc{a}) = f$, then the arc $b$ preceding $a$ in the rotation of $v$ is incident to $f$. \label{fig:NTRM=NRM}}
\end{figure}

The matrix $C := N^TM$ is the \textsl{vertex-face incidence matrix} of $X$; the $(v,f)$-entry of $C$ is equal to the number of times that the vertex $v$ appears on the facial walk of $f$. This insight gives an alternative `proof' for Proposition \ref{prop:NLM_props}(iii): if we had defined the facial walks to be going in the anticlockwise direction, then $RM$ would have been the arc-face incidence matrix of the map, but the entries of $C$ don't depend on the orientation, hence $N^TM = N^TRM$.

An orientable map $X$ has \textsl{incidence multiplicity} $\alpha$ if whenever a vertex appears on the facial walk of a face, it appears on that face exactly $\alpha$ times. Equivalently, $X$ has incidence multiplicity $\alpha$ if all non-zero entries of $C$ are equal to $\alpha$. If $\alpha = 1$ and all facial walks have length at least $3$, then we say that $X$ is \textsl{circular}; in this case, every facial walk is a cycle in the graph underlying $X$. For the example $X_2$ of the embedded digon, we have
\[
C = \kern-.4em
\begin{array}{cc}
  & \begin{array}{cc}f_1 \!\!& \!\! f_2 \end{array} \\
\begin{array}{c} v_1 \kern-1.5em \\ v_2  \kern-1.5em \end{array} & \left[\begin{array}{cc}1 & 1 \\ 1 & 1 \end{array}\right]
\end{array},
\]
and thus $X_2$ is an example of a map with incidence multiplicity $1$, but which is not circular.

By construction, each of the matrices $N$ and $M$ has pairwise orthogonal columns. For the vertex-face quantum walk we need to consider the normalized versions of these incidence matrices; we define 
\[
\Nhat := ND^{-\frac{1}{2}} \quad \text{and} \quad \Mhat := M\Delta^{-\frac{1}{2}}
\]
to be the \textsl{normalized arc-vertex} and \textsl{arc-face incidence matrices}, respectively. The sets of columns of $\Nhat$ and $\Mhat$ form orthonormal bases for the column spaces of $N$ and $M$ respectively. We obtain the following result.

\begin{corollary}\label{cor:inci-properties}
The following properties hold:
\begin{enumerate}[(i)]
\item $\Nhat^T\Nhat = \Mhat^T\Mhat = I$;
\item $\Nhat^TR\Mhat = \Nhat^T\Mhat$.
\end{enumerate}
\end{corollary}

We also define 
\[
\Chat := \Nhat^T\Mhat = D^{-\frac{1}{2}}C\Delta^{-\frac{1}{2}}
\]
to be the \textsl{normalized vertex-face incidence matrix}.

An \textsl{automorphism} of a map $X$, orientable or non-orientable, is a permutation of the flags of $X$ that preserves all incidences between flags, vertices, edges and faces. We denote the group of automorphisms of $X$ by $\Aut(X)$. Every automorphism is completely determined by the image of any single flag. 
%This implies that the action of $\Aut(X)$ on $\cF$ is free. T
Thus if the action of $\Aut(X)$ on $\cF$ is transitive, it is regular. In that case, we say that $X$ is a \textsl{(fully) regular map}. Specifically, if such a map $X$ is orientable, it is called \textsl{reflexible}. If $X$ is orientable and the action $\Aut(X)$ on the flag has, not one, but two orbits, which are the sets of clockwise and anticlockwise flags, then $X$ is a \textsl{chiral map}. An \textsl{orientably-regular map} is an orientable map that is either reflexible or chiral. Though some definitions vary among the literature, our nomenclature is consistent with the census of regular maps\cite{Con2009, Con2012}, as is used in Section \ref{sec:computations}.

Each automorphism of $X$ induces permutations of the vertices, edges and faces, preserving incidences. We record these in the following proposition, for use in later sections.

\begin{proposition}
\label{prop:regmaps_perms}
Suppose that $\pi$ is an automorphism of $X$ and write $\pi_\cA$, $\pi_V$, $\pi_E$ and $\pi_F$ for the permutation matrices that correspond to the action of $\pi$ on the sets of clockwise flags, vertices, edges and faces of $X$ respectively. Then
\begin{enumerate}[(i)]
    \item $\pi_\cA N = N \pi_V$;
    \item $\pi_\cA L = L \pi_E$; and 
    \item $\pi_\cA M = M \pi_F$. \qed 
\end{enumerate}
\end{proposition}

Clearly, the actions of $\Aut(X)$ on the sets $V$, $E$ and $F$ are transitive if $X$ is a rotary map.

With these preliminaries in mind, we will retain the definitions of $M,N, R$ and $\cA$ for a map $X$ for the rest of the paper, unless specifically stated otherwise.

\section{Vertex-face quantum walk}\label{sec:vxfacewalkdefn}

In this section, we define the vertex-face quantum walk and state the existing results.

Suppose that $X$ is an orientable map, and that $N$ is its arc-vertex incidence matrix and $M$ its arc-face incidence matrix. Let $Q,P \in \cx^{\cA \times \cA}$ be the orthogonal projections onto the column spaces of $N$ and $M$ respectively. Note that we can write
\[
Q = \Nhat\Nhat^T \quad \text{and} \quad P = \Mhat\Mhat^T,
\]
where $\Nhat$ and $\Mhat$ are the respective normalized incidence matrices, because the columns of $\Nhat$ and $\Mhat$ form respective orthonormal bases for $\col(N)$ and $\col(M)$. We speak of the column spaces of $N$ and $\Nhat$ interchangeably, and do the same for $M$ and $\Mhat$. For readability we will often, if possible, use just $N$ and $M$ instead of their normalized versions. For instance, if $X$ is a type $(k,d)$ map, we can write
\[
Q = \frac{1}{d}NN^T \quad \text{and} \quad P = \frac{1}{k}MM^T.
\]
For our example of the embedded digon $X_2$, as defined in the previous section, the matrices $Q$ and $P$ are given as follows:
\[
Q = \frac{1}{2}\bmat{1 & 0 & 1 & 0 \\ 0 & 1 & 0 & 1 \\ 1 & 0 & 1 & 0 \\ 0 & 1 & 0 & 1} \quad \text{and} \quad P =  \frac{1}{2}\bmat{1 & 0 & 0 & 1 \\ 0 & 1 & 1 & 0 \\ 0 & 1 & 1 & 0 \\ 1 & 0 & 0 & 1}.
\]

Let $U \in \cx^{\cA \times \cA}$ be the unitary matrix defined by
\[
U = (2P - I)(2Q - I).
\]
That is, $U$ is the product of the reflections through the column spaces of $M$ and $N$. The \textsl{vertex-face (quantum) walk} on $X$, given an initial state $\ket{\psi} \in \cx^{\cA}$, is given by the sequence
$(U^t \ket{\psi})_{t \in \ints_{\geq 0}}$,
and $U$ the \textsl{transition matrix} of the vertex-face walk on $X$, or `the transition matrix of $X$' for short. Note that all of the matrices involved have real entries. The transition matrix for the dual map is given by
\[
(2Q - I)(2P - I) = U^T,
\]
which is the inverse of $U$, since $U^* = U^T$. We note that we have made an arbitrary decision, following Zhan \cite{Zha2020} to use the clockwise flags; one can derive a more formal correspondence between the use of clockwise and anticlockwise flag using a direct part (ii) of Corollary \ref{cor:inci-properties}. 

For the map $X_2$ whose arcs are shown in Figure \ref{fig:digon_arcs}, we can compute that
\[
U = I \otimes \bmat{0 & 1 \\ 1 & 0}
\]
and thus  $U^2 = I$; in this case the vertex-face walk will alternate between two states. Different characterisations of maps for which the transition matrix $U$ satisfies $U^2 = I$ are given in Lemma \ref{lem:U^2=I}.

To give a more visual example, we consider the Heawood graph embedded on the torus, whose dual is $K_7$. See Figure \ref{fig:heawood}. Let $\ket{\psi}$ be the state consisting of the uniform superposition of the out-going arcs of vertex $6$; that is $\ket{\psi} = \Nhat \Ze_6$. Similarly, let $\ket{\phi} :=\Nhat \Ze_4 $ be the state consisting of the uniform superposition of the out-going arcs of vertex $4$. The probability of measuring at $\ket{\phi}$ at time $t$ with initial state $\ket{\psi} $ is given by $|| \langle \phi | U^t \ket{\psi} ||^2$. The plot of the right side of Figure \ref{fig:heawood} shows this probability for $t =0,\ldots, 49$.

\begin{figure}[htb]
    \centering
    \begin{subfigure}{6cm}
    \includegraphics[width=5.5cm,height=5.5cm]{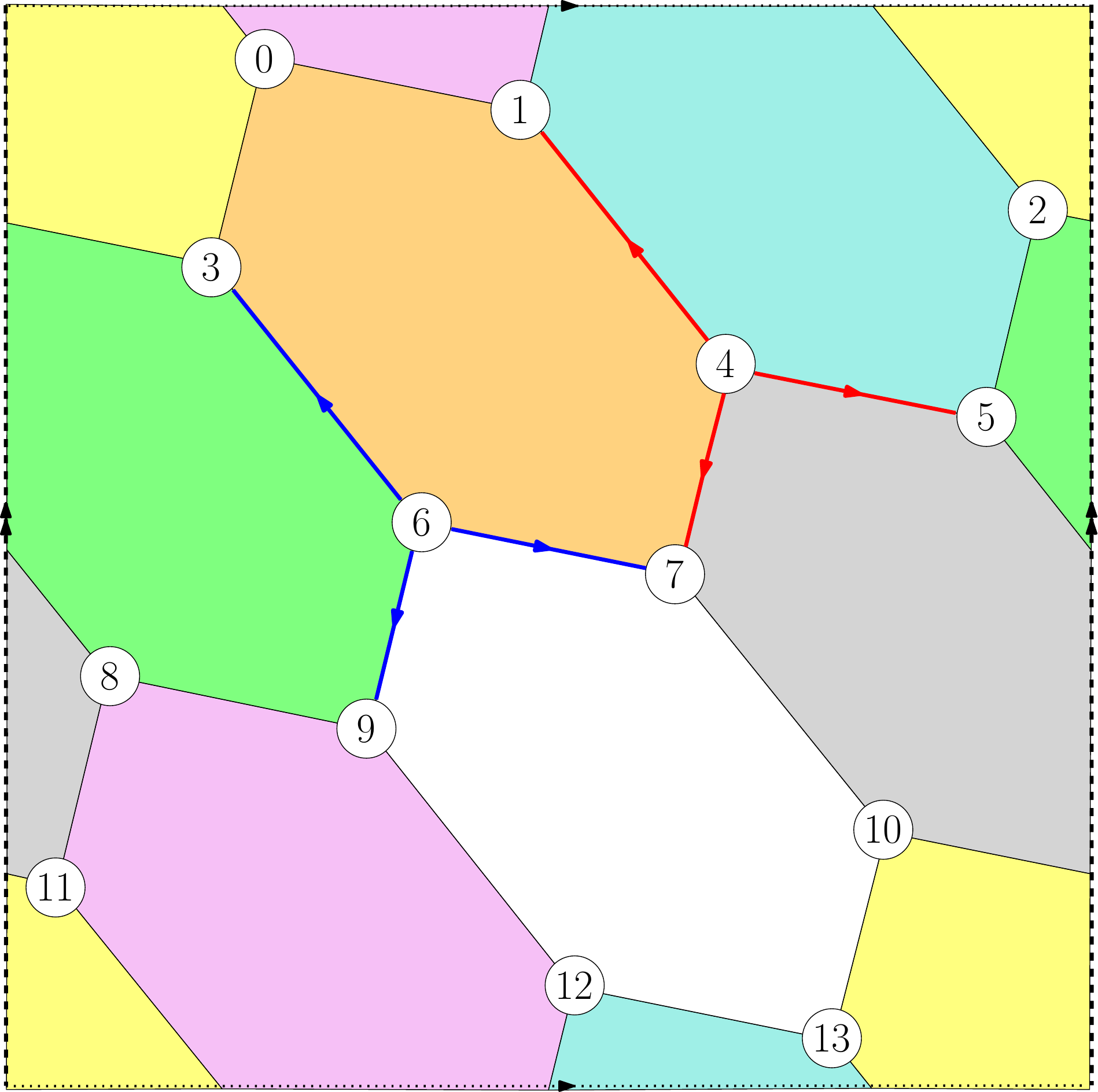}
    \end{subfigure}
    \begin{subfigure}{10cm}
    \includegraphics[width=9.5cm]{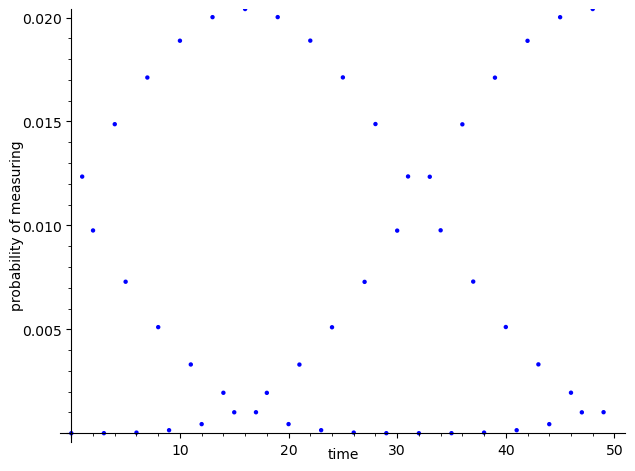}
    \end{subfigure}
    \caption{Heawood graph with the out-arcs at vertices $6$ and $4$ distinguished.  \label{fig:heawood}}   
\end{figure}

The following result about the $1$ and $(-1)$-eigenspaces of $U$ is due to \cite{MagNayRol2011} for general quantum walks and appears as \cite[Theorems 3.1, 3.3]{Zha2020} for vertex-face walks.  Recall that $C = N^TM$. 

\begin{theorem}\cite{Zha2020}
\label{thm:1_eigenspaces}
Let $U$ be the transition matrix for the vertex-face walk of an orientable map $X$.
\begin{enumerate}[(i)]
\item The $1$-eigenspace of $U$ is
\[
(\col(M) \cap \col(N)) \oplus (\ker(M^T) \cap \ker(N^T))
\]
and has dimension $|E| + 2g$. The first subspace in this direct sum is
\[
\col(M) \cap \col(N) = \vecspan\{\ones\}.
\]
\item The $(-1)$-eigenspace of $U$ is
\[
(\col(M) \cap \ker(N^T)) \oplus (\ker(M^T) \cap \col(N))
\]
and has dimension $|V| + |F| - 2\rk(C)$.
\end{enumerate}\qed
\end{theorem}

We see that the all-ones vector $\ones$ is always in the $1$-eigenspace of $U$, so the $1$-eigenspace is non-trivial. Note that the expression for the dimension of the $(-1)$-eigenspace is always nonnegative. This space is trivial only if the number of vertices of the map equals the number of faces:

\begin{corollary}
\label{cor:U_-1_not_eig}
If $-1$ is not an eigenvalue of $U$, then $|V| = |F|$.
\end{corollary}
\proof
If $-1$ is not an eigenvalue of $U$, then the dimension of its $(-1)$-`eigenspace' is $0$. By Theorem \ref{thm:1_eigenspaces}(ii), this implies that
\[
|V| + |F| = 2\rk(C).
\]
Since the rank of $C$ is at most $\min(|V|,|F|)$, we must have $|V| = |F|$. \qed

As $U$ is real and unitary, eigenvalues other than $\pm 1$ come in conjugate pairs and lie on the unit circle. (By Theorem \ref{thm:1_eigenspaces}, there are $\rk(C) - 1$ such pairs, with multiplicity). The following results describe how these eigenvalues arise from a smaller matrix whose rows and columns are indexed by the vertex set:
$
\Chat \Chat^T = \Nhat^T P\Nhat$, where $\Chat = \Nhat^t \Mhat$.
As shown in \cite{Zha2020}, the eigenvalues of $U$ can be expressed in terms of the eigenvalues of this matrix. To better facilitate our results, we give a  restatement of their result as a decomposition of the space $\cx^{\cA}$ into root spaces, along with the minimal polynomials of $U$ over each root space, each each of which ahs degree at most $2$. 

\begin{theorem}\cite{Zha2020}
\label{thm:orth_decomposition}
Let $\{\Zv_i\}_{i = 1}^{|V|} \subset \cx^{V}$ be an orthogonal eigenbasis of  $\Chat\Chat^T$, with corresponding eigenvalues $\hat{\lambda}_1, \ldots, \hat{\lambda}_{|V|}$. Then $\cx^\cA$ can be decomposed into a direct sum of orthogonal subspaces as follows:
\[
\cx^\cA = \cK \oplus \cW \oplus \bigoplus_{i : \hat{\lambda}_i \notin \{0,1\}} \cJ_i,
\]
where $\cK$ is the $1$-eigenspace and $\cW$ the $(-1)$-eigenspace of $U$, and where
\[
\cJ_i = \vecspan\{\Nhat \Zv_i, P\Nhat \Zv_i\}, \quad i=1,\ldots,|V|.
\]
If $\hat{\lambda}_i \neq 0,1$, the minimal polynomial of $U$ over $\cJ_{i}$ is given by
\[
p_i(t) = t^2 - (4\hat{\lambda}_i -2)t + 1. \qedhere
\]
\end{theorem}

Note that if $X$ is a type $(k,d)$ map (i.e.\ every vertex has degree $d$ and every face has degree $k$, then $\Chat\Chat^T = CC^T/(kd)$. 
As an example, let $X$ be the $2 \times 3$ grid embedded in the torus. For this map, the eigenvalues of $CC^T$ are as follows:
$\lambda_1 = \lambda_2 = \lambda_3 = 0$,  $\lambda_4 = \lambda_5 = 4$, and $\lambda_6 = 16$.
The eigenvalues of $\Chat\Chat^T$ are then given by $\hat{\lambda}_i = \lambda_i/16$ for all $i$. Then $\lambda_4,\lambda_5 \in (0,1)$, so
\[
\cx^{\cA} = \cK \oplus \cW \oplus \cJ_4 \oplus \cJ_5,
\]
where the $1$-eigenspace $\cK$ of $U$ has dimension $14$ and the $(-1)$-eigenspace $\cW$ has dimension 6. The spaces $\cJ_4$ and $\cJ_5$ are two-dimensional subspaces over which $U$ has minimal polynomial
\[
p_4(t) = p_5(t) = t^2 + t + 1.
\]
Indeed, the dimensions of these subspaces add up to $|\cA| = 24$.

Recall that if a type $(k,d)$ map has incidence multiplicity $\alpha$, then whenever a vertex $v$ is incident to a face $f$, that $v$ is traversed exactly $\alpha$ times by the facial walk of $f$. In the following lemma, we generalize Lemma 2.3 from \cite{Zha2020} from circular embeddings to type $(k,d)$ maps of any incidence multiplicity. 

\begin{lemma}\label{lem:Utrace}
If $X$ is a type $(k,d)$ map with an incidence multiplicity $\alpha$, then every diagonal entry of $U$ is equal to
\[
\frac{4\alpha}{kd}  - \frac{2}{k} - \frac{2}{d} + 1.
\]
Moreover,
\[
\tr(U) = \frac{4\alpha |V|}{k} - 2|F| - 2|V| + |\cA|.
\]
\end{lemma}

\begin{proof}
We have $Q = \frac{1}{d} NN^T$ and $P = \frac{1}{k}MM^T$ because $X$ has type $(k,d)$. It is not difficult to see that for all $a,b \in \cA$:
\[
Q_{a,b} = \begin{cases}
\frac{1}{d} &\text{if $v(a) = v(b)$}; \\
0 &\text{otherwise,}
\end{cases} \quad \text{and} \quad
P_{a,b} = \begin{cases}
\frac{1}{d} &\text{if $f(a) = f(b)$}; \\
0 &\text{otherwise,}
\end{cases}
\]
where $v(c)$ and $f(c)$ are respectively the vertex and face incident with the arc $c$. This implies that
\[
(PQ)_{a,a} = \sum_{c \in \cA} P_{a,c}Q_{c,a} = \frac{\text{\#arcs incident to both $v(a)$ and $f(a)$}}{kd} = \frac{\alpha}{kd}.
\]
since $U$ can be written as
\[
U = 4PQ - 2P - 2Q + I,
\]
we find that for $a \in \cA$:
\[
U_{a,a} = \frac{4\alpha}{kd}  - \frac{2}{k} - \frac{2}{d} + 1
\]
and then
\[
\tr(U) = |\cA|\cdot U_{a,a} = \tr(U) = \frac{4\alpha |V|}{k} - 2|F| - 2|V| + |\cA|,
\]
by using that $|\cA| = |F|\cdot k = |V|\cdot d$.
\end{proof}

\section{Perfect state transfer} 
\label{sec:pst}

In this section, we define perfect state transfer for the vertex-face walk and give necessary and sufficient conditions for it to occur. In order to do this, we first work in a more general setting for discrete-time quantum walks and define an auxiliary sequence of matrices describing the walk, which satisfy a Chebyshev recurrence. Specifying to the vertex-face walk, we culminate in some fundamental properties in Theorem \ref{thm:pst_reverse_periodicity}. 

In the two-reflection model of a discrete-time quantum walk, the transition matrix $U$ is of the form
\[
U = (2WW^* - I)(2VV^* - I),
\]
where $W \in \cx^{k,n}$ and $V \in \cx^{k,m}$ are matrices with orthonormal columns (i.e. $W^*W = V^*V = I$). 

We now define an auxiliary sequence of matrices corresponding to $U$, which will allow us to work with Hermitian matrices. For all $t \in \ints_{\geq 0}$, let $B_t$ be the matrix that describes the action of $U^t$ on the column space of $V$, with respect to the orthonormal basis formed by the columns of $V$; that is,
\[
B_t := V^* U^t V.
\]
Whereas $U$ is not usually Hermitian, the matrix $B_t$ is Hermitian for all $t$; $B_0 = I$ and, since $(2VV^* - I)$ acts as the identity on $V$, for $t > 0$, we can write
\[
\begin{split}
B_t &= V^* \left( (2WW^* - I)(2VV^* - I) \right)^t V \\
&= V^*\left((2WW^* - I)(2VV^* - I)\right)^{t-1} (2WW^* - I) V
\end{split}
\]
which is clearly Hermitian. In the following theorem, we show that the sequence $\{B_t\}_{t=0}^{\infty}$ satisfies the same recurrence as the Chebyshev polynomials of the first kind; for more background on Chebyshev polynomials, we refer to \cite{Riv1974}. 

\begin{theorem}\label{thm:B_Chebyshev}
For all $t \in \ints_{\geq 0}$, we have
\[
B_t = T_t(B_1),
\]
where $T_t$ is the $t$-th Chebyshev polynomial of the first kind.
\end{theorem}

\begin{proof}
Since $B_0 = I$, 
it suffices to show that $B_t$ satisfies the recursion
\[
B_{t+1} = 2B_tB_1 - B_{t-1}.
\]
In order to do so, define $D_t = V^* U^t WW^* V$, so that $D_0 = V^* WW^* V$. We claim that $B_t$ and $D_t$ satisfy
\begin{equation}\label{eq:btrecur}
    \begin{cases}
B_{t+1} = 2D_{t} - B_t \\
D_{t+1} = 2B_{t+1}D_0 - D_{t}
\end{cases}
\end{equation}
for all $t \geq 1$. By applying the claim three times, we find
\[
\begin{split}
B_{t+1} &= 2D_{t} - B_t \\
        &= 4B_t D_0 - 2D_{t-1} - B_t \\
        &= 4B_t D_0 - (B_{t} + B_{t-1}) - B_{t} \\
        &= 2B_t (2D_0 - I) - B_{t-1} \\
        &= 2B_t B_1 - B_{t-1},
\end{split}
\]
for all $t \geq 1$, as desired. It remains to prove \eqref{eq:btrecur}. Since $(2VV^* - I)V = V$, we find that
\[
\begin{split}
B_{t+1} &= V^* U^{t}(2WW^* - I)V
        = 2V^* U^{t}WW^*V - V^* U^{t}V
        = 2D_{t} - B_{t}.
\end{split}
\]
Finally, we have
\[
\begin{split}
D_{t+1} &= V^* U^t (2WW^* - I)(2VV^* - I) WW^* V \\
        &= 2V^* U^t (2WW^* - I)VV^*WW^* V - V^* U^t (2WW^* - I)WW^*V \\
        &= 2V^* U^{t+1} VV^* WW^* V - V^* U^t WW^* V \\
        &= 2B_{t+1}V^* WW^* V - D_{t} \\
        &= 2B_{t+1}D_0 - D_{t},
\end{split}
\]
where for the third equality, we used for the first term that $V = (2VV^* - I)V$, and for the second term that $(2WW^* - I)W = W$.
\end{proof}

Denote by $\Ze_u$ the $u$-th standard basis vector, so that $V\Ze_u$ is the $u$-th column of $V$. We now use the recurrence to show symmetry for state transfer from  $V \Ze_u $ to $V \Ze_v $ and $V \Ze_v $ to $V \Ze_u $. 

\begin{lemma}
\label{lem:pst_B_General}
We have $U^{\tau} V \Ze_u = V\Ze_v$  if and only if
$B_\tau (u,v) = B_\tau(v,u) = 1.$
\end{lemma}
\begin{proof}
If $U^{\tau} V \Ze_u = V\Ze_v$, then by multiplying both sides by $\Ze_v^*V^*$ on the left, we obtain
\[
B_\tau (u,v) = B_\tau(v,u) = \Ze_v^*B_\tau \Ze_u = 1
\]
since $B_\tau$ is Hermitian. Conversely, if $\Ze_v^* B_\tau \Ze_u = 1$, then
\[
\inprod{V\Ze_v}{U^\tau V \Ze_u} = 1,
\]
and as both $V \Ze_v$ and $U^\tau V \Ze_u$ have unit length, this implies $U^\tau V \Ze_u = V \Ze_v$.
\end{proof}

Considering the matrices $B_i$ allows us to connect when the quantum walk takes a specific uniform superposition at $u$ to another vertex $v$ with an algebraic property of the graph, which we will now define. Let $M$ be a Hermitian matrix with rows and columns indexed by a set $\Omega$, and let 
\[
M = \sum_{r= 0}^d \theta_r E_r
\]
be the spectral decomposition of $M$. Then $u,v \in \omega$ are said to be \textsl{strongly cospectral with respect to $M$}
if 
\[
E_r \Ze_v = \pm E_r \Ze_u
\]
for all $r = 0,\ldots, d$. 
Strongly cospectral vertices have been previously studied in the context of continuous-time quantum walks, where $M$ is the adjacency matrix of a graph, see \cite{BanCouGod2017}. We obtain the following directly from Theorem \ref{thm:B_Chebyshev}. 

\begin{corollary}\label{cor:statetransferGeneral}
Let $u,v \in \{1,\ldots,m\}$. Then the following hold:
\begin{enumerate}[(i)]
    \item $U^{\tau} V \Ze_u = V\Ze_v$ if and only if $U^{\tau} V\Ze_v = V\Ze_u$
    \item If $U^{\tau} V \Ze_u = V\Ze_v$, then $u$ and $v$ are strongly cospectral with respect to $B_d$ for all divisors $d$ of $\tau$.
\end{enumerate}
\end{corollary}

\begin{proof}
Part (i) follows from Lemma \ref{lem:pst_B_General} and the fact that $B_t$ is symmetric for all $t$. 

For (ii), let $d$ be any positive integer that divides $\tau$. Let the following be the spectral decomposition of $B_d$:
\[
B_d = \sum_{\theta} \theta F_\theta .
\]
By Theorem \ref{thm:B_Chebyshev} and by the properties of of $T_t$ under composition, that 
\[
B_\tau = T_\tau(B_1) = T_{\ell}(T_{d}(B_1)) = T_\ell(B_d),
\]
where $\ell = \tau / d$. Since $B_{\tau} \Ze_u = \Ze_v$, we have for every eigenvalue $\sigma$ of $B_d$:
\[
F_\sigma \Ze_v = F_\sigma B_\tau \Ze_u 
=  F_\sigma T_\ell(B_d) \Ze_u 
= F_\sigma \sum_{\theta}T_\ell(\theta) F_\theta \Ze_u 
= T_\ell(\sigma) F_\sigma \Ze_u .
\]
Repeating the argument for $F_\sigma \Ze_u$, since $B_{\tau} \Ze_v = \Ze_u$, we see that $F_\sigma \Ze_u = T_\ell(\sigma) F_\sigma \Ze_v$. 
Thus, we find that 
\[
F_\sigma \Ze_u = T_\ell(\sigma) F_\sigma \Ze_v = T_\ell(\sigma)^2 F_\sigma \Ze_u
\]
and thus $T_\ell(\sigma) \in \{\pm 1\}$ unless $F_{\sigma} \Ze_u = F_{\sigma} \Ze_v = 0$, and the result follows. 
\end{proof}

Now we will apply this idea to the vertex-face walk. 
Recall that for a map $X$, the transition matrix $U$ of $X$ for the vertex-face walk is defined by
\[
U = (2P - I)(2Q-I),
\]
where $P = \Mhat\Mhat^T$ and $Q = \Nhat\Nhat^T$ are the projectors onto the column spaces of $M$ and $N$ respectively, and $\Mhat$ and $\Nhat$ are the normalized arc-face and arc-vertex incidence matrices. When we consider the evolution of a quantum system, we usually take the initial state to be a uniform superposition of all the arcs incident to some vertex $u$; in particular, we consider the state $N \Ze_u$, where $\Ze_u \in \cx^V$ is the elementary basis vector indexed by $u$.

For vertices $u,v \in V$,  if 
\begin{equation}
\label{eq:uv_PST}
U^\tau \Nhat \Ze_u = \Zx,
\end{equation}
where $\Zx \in \cx^\cA$ is a unit length vector that satisfies $\Nhat\Ze_w \circ \Zx = 0$ for all $w \neq v$ (i.e.\ $\Zx$ is any superposition of the arcs incident to $v$), then, if $u,v$ are distinct, we say that there is \textsl{perfect state transfer from $u$ to $v$ at time $\tau \in \ints_{>0}$} and if $u=v$, we say that $X$ is  \emph{periodic at the vertex $u$} at time $\tau$. 
For convenience, we will write \textsl{$uv$-PST} for perfect state transfer from $u$ to $v$. 

By the following lemma, we can simplify the expression in \eqref{eq:uv_PST} for such maps, using the same ideas as \cite[Lemma 3.2.1]{GodZha2019} for the arc-reversal Grover walk on $d$-regular graphs. Note that, in the following lemma, the underlying graph is not necessarily regular, but the two concerning vertices must have the same degree.

\begin{lemma}
\label{lem:uv_PST_simplified}
Let $X$ be a map and $u,v$ be vertices of $X$ with degree $d$. Then there is $uv$-PST at time $\tau$ if and only if
\begin{equation}
\label{eq:uv_PST_simplified}
U^\tau \Nhat \Ze_u = \Nhat \Ze_v.
\end{equation}
\end{lemma}
\begin{proof}
It is clear that if (\ref{eq:uv_PST_simplified}) holds, then there is $uv$-PST at time $\tau$ by definition. For the converse, suppose that there is $uv$-PST at time $\tau$, so that (\ref{eq:uv_PST}) holds for some appropriate $\Zx$. Note  that $\Nhat\Ze_u = \frac{1}{\sqrt{d}}N\Ze_u$ and that $\ones_{\cA}$ is an eigenvector of $U^T$ with eigenvalue $1$. We have
\[
\sqrt{d} = \inprod{\ones_{\cA}}{\Nhat \Ze_u} 
= \inprod{(U^T)^\tau \ones_{\cA}}{\Nhat \Ze_u} 
= \inprod{\ones_{\cA}}{U^\tau \Nhat \Ze_u} 
= \inprod{\ones_{\cA}}{\Zx}.
\]
Since $\Zx$ takes non-zero entries only on the arcs incident with $v$, we have that 
$\inprod{\ones_{\cA}}{\Zx} = \inprod{N\Ze_v}{\Zx} $
whence we obtain that 
$ \sqrt{d} = \sqrt{d} \inprod{\Nhat \Ze_v}{\Zx}$.
Since both $\Nhat \Ze_v$ and $\Zx$ have length $1$, the equality $\inprod{\Nhat \Ze_v}{\Zx} = 1$ implies that $\Zx = \Nhat \Ze_v$.
\end{proof}

We note that, in the general case, there can only be $uv$-PST if the degree of $v$ is at least the degree of $u$; in the proof of Lemma \ref{lem:uv_PST_simplified}, the general case yields
\[
\sqrt{d(u)} = \sqrt{d(v)} \inprod{\Nhat\Ze_v}{\Zx} \leq  \sqrt{d(v)}
\]
 by Cauchy-Schwarz. We will restrict our attention to perfect state transfer between vertices of equal degree and we can take \eqref{eq:uv_PST_simplified} to be the definition of $uv$-PST at time $\tau$. For periodicity at a vertex $u$, \eqref{eq:uv_PST_simplified} (with $u = v$) is equivalent to the original definition for any map, by Lemma \ref{lem:uv_PST_simplified}.

As in the general case, we will consider, for $t \in \ints_{\geq 0}$, the matrix
\[
B_{t} = \Nhat^T U^t \Nhat. 
\]
It follows directly from Lemmas \ref{lem:pst_B_General} and \ref{lem:uv_PST_simplified} that 
there is $uv$-PST at time $\tau$ if and only if
\[
B_\tau (u,v) = B_\tau(v,u) = 1,
\]
and there is periodicity at $u$ at time $\tau$ if and only if 
$B_\tau (u,u) = 1$.

Though $B$ is not, in general, a stochastic matrix, we note that 
for all $t \in \ints_{\geq 0}$, the vector
$\Zw =  (\sqrt{d(v)})_v$
is an eigenvector for $B_t$ with eigenvalue $1$. In particular, if the graph underlying the map is $d$-regular, every row of $B_t$ sums to $1$. In this setting, we have that $B_0 = I$, $B_1 = 2\Chat\Chat^T - I$, and 
$B_t = T_t(B_1)$
where $T_t$ is the $t$-th Chebyshev polynomial of the first kind, by Theorem \ref{thm:B_Chebyshev}. 

We now apply our results from the general setting and we  establish some fundamental properties of  perfect state transfer in the vertex-face walk. 

\begin{theorem}
\label{thm:pst_reverse_periodicity}
Consider the vertex-face quantum walk on a map $X$, and let $u,v \in V(X)$ be distinct vertices of $X$. Assume that there is $uv$-PST at time $\tau \in \ints_{>0}$. Then
\begin{enumerate}[(i)]
\item there is $vu$-PST at time $\tau$;
\item there is periodicity at both $u$ and $v$ at time $2\tau$;
\item there does not exist a vertex $w$, distinct from $u$ and $v$, such that there is $uw$-PST at any time; and
\item $u$ and $v$ are strongly cospectral with respect to $B_d$ for all divisors $d$ of $\tau$.
\end{enumerate}
\end{theorem}
\begin{proof}
Clearly, (ii) follows directly from (i). The property (i) follows from Lemma \ref{lem:pst_B_General} and the fact that $B_\tau$ is symmetric. For (iii), assume that $\tau$ is the smallest time at which there is $uv$-PST. If there is some $w \neq u,v$ for which there is $uw$-PST at some time $\tau'$, where $\tau'$ is also minimal, then there is periodicity at $u$ both at time $2\tau'$ and at time $2\tau$. The minimality of $\tau$ and $\tau'$ ensures that both $\tau' < 2\tau$ and $\tau < 2\tau'$. Assume without loss of generality that $\tau' < \tau$. Then there is periodicity at $u$ at time $2(\tau - \tau')$, but
\[
2(\tau - \tau') < 2\tau - \tau = \tau,
\]
contradicting the minimality of $\tau$.

Part (iv) follows directly from applying Corollary \ref{cor:statetransferGeneral} with $V= \Nhat$ and $W=\Mhat$. 
\end{proof}

We note that, in particular, if there is $uv$-pst at any time, the vertices $u$ and $v$ are strongly cospectral with respect to $B_1$. We remark that the proof of (iv) also implies that $T_d(\sigma) = \pm 1$ if $F_\sigma \Ze_u \neq 0$, so the \textsl{eigenvalue support} of $u$ (and also of $v$) is the set $\{\pm 1\}$. 

\section{Periodic maps}\label{sec:periodicity}

We have seen in Theorem \ref{thm:pst_reverse_periodicity} of the previous section that $uv$-PST results in periodicity at $u$ and $v$ at twice the time. Thus if the vertex set partitions into pairs such that perfect state transfer occurs between every pair at time $\tau$, then there is periodicity at every vertex at time $2\tau$. If the automorphism group of the map acts transitively upon the vertex set and there is $uv$-PST for some pair of vertices, then there must be perfect state transfer everywhere (see Theorem \ref{thm:transitive_pst}). Motivated by this, we will turn our attention to maps where periodicity occurs at every vertex. 

Now we will proceed with some formal definitions. 
Let $X$ be an orientable map. If there is periodicity at every vertex at time $\tau$, i.e. if $U^\tau \Nhat = \Nhat$, we say that $X$ is \textsl{periodic at time $\tau$}. Equivalently, $X$ is periodic at time $\tau$ if $U^\tau$ acts as the identity on $\col(N)$. We call $\tau$ the \textsl{period} of $X$ if $\tau$ is minimal.
We start by showing that if both $X$ and dual map $X^*$ are periodic at time $\tau$, then $U^\tau$ is the identity matrix:

\begin{lemma}
\label{lem:map_and_dual_periodic}
If both $U^\tau \Nhat = \Nhat$ and $U^\tau \Mhat = \Mhat$ for some $\tau > 0$, then $U^\tau = I$.
\end{lemma}
\begin{proof}
Recall from Theorem \ref{thm:orth_decomposition} that we can write
\begin{equation}
\label{eq:periodicity_Ueig_decomposition}
\cx^\cA = \cK \oplus \cW \oplus \bigoplus_{i : \hat{\lambda}_i \notin \{0,1\}} \cJ_i,
\end{equation}
where $\cK$ and $\cW$ are the respective $1$- and ($-1$)-eigenspaces of $U$, and where for every eigenvector $\Zv_i$ of $\Chat\Chat^T$ with eigenvalue $\hat{\lambda}_i$, the space $\cJ_i$ is defined by
\[
\cJ_i = \vecspan\{P\Nhat\Zv_i, \Nhat\Zv_i\}.
\]
It suffices to show that if $U^\tau \Nhat = \Nhat$ and $U^\tau \Mhat = \Mhat$, then $U^\tau$ acts as the identity on each of the subspaces in \eqref{eq:periodicity_Ueig_decomposition}. By definition, $U^\tau$ acts as the identity on $\cK$. Furthermore, by Theorem \ref{thm:1_eigenspaces}, we can write 
\[
W = (\col(M) \cap \ker(N^T)) \oplus (\ker(M^T) \cap \col(N)),
\]
which is a subspace of $\col(M) + \col(N)$. In addition, every $\cJ_i$ is a subspace of $\col(M) + \col(N)$ as well. As $U^\tau$ acts as the identity on $\col(M) + \col(N)$, this concludes the proof. 
\end{proof}

In the following theorem, we establish the connection between periodic maps and those where $U^{\tau} =I$;  we see that, in many cases, periodicity of the map at time $\tau$ implies that $U^\tau = I$, for instance, when $\tau$ is even.

\begin{theorem}
\label{thm:periodic_implies_id}
The transition matrix satisfies $U^{\tau} = I$ if and only if the map is periodic at time $\tau \in \ints_{>0}$ where $\tau$ satisfies one of the following:
\begin{enumerate}[(a)]
    \item $\tau$ is even, or 
    \item $\tau$ is odd and $|V| = |F|$. 
\end{enumerate}
Further, if the map is periodic at time $\tau$, then $U^{2\tau} = I$. 
\end{theorem}

\begin{proof} Define 
\[
\tilde{U} := U^\tau (2Q-I) = ((2P - I)(2Q - I))^{\tau-1} (2P - I),
\]
so that also $U^\tau = \tilde{U}(2Q-I)$. We see that $U^\tau = I$ if and only if $\tilde{U} = 2Q - I$. Moreover, as $\tilde{U}$ is both unitary and symmetric, it is an involution. In particular, both $\tilde{U}$ and $2Q - I$ are symmetric with eigenvalues in $\{-1,1\}$; we can conclude that
\begin{equation}
\label{eq:Utilde_eigenspace}
U^\tau = I \quad \iff \quad \text{$\tilde{U}$ and $2Q-I$ have the same $1$-eigenspace.}
\end{equation}

Now assume that the map is periodic at time $\tau \in \ints_{>0}$. Then
\[
\Nhat = U^\tau \Nhat = \tilde{U}(2Q - I) \Nhat = \tilde{U} \Nhat,
\]
so the $1$-eigenspace of $\tilde{U}$ contains $\col(N)$, which is the $1$-eigenspace of $2Q-I$. By \eqref{eq:Utilde_eigenspace}, it is now sufficient to show that the multiplicity of the eigenvalue $1$ is the same for $\tilde{U}$ and $2Q-I$. Equivalently, we can show that these matrices have equal trace. By using the cyclic property of the trace, and the fact that $2P - I$ and $2Q -I$ are involutions, the trace of $\tilde{U}$ can be reduced as follows:
\[
\begin{split}
\tr(\tilde{U}) &= \tr\left[((2P - I)(2Q - I))^{\tau-1} (2P - I)\right] \\
&= \tr\left[(2Q - I)((2P - I)(2Q - I))^{\tau-2}\right] \\
&\phantom{..}\vdots  \\
&= \begin{cases}
\tr(2Q - I) \quad &\text{if $\tau$ even;}\\
\tr(2P - I) \quad &\text{if $\tau$ odd.}
\end{cases}
\end{split}
\]
This proves that $U^\tau = I$ if $\tau$ is even, or if $\tau$ is odd and $\tr(2P-I) = \tr(2Q-I)$. The latter happens exactly if $P$ and $Q$ have equal rank, i.e.\ if $|F| = |V|$.

Conversely, assume that $U^\tau = I$. Then the map is certainly periodic at time $\tau$. Moreover, $\tilde{U} = 2Q-I$, so if $\tau$ is odd, then
\[
\tr(2Q - I) = \tr(\tilde{U}) = \tr(2P - I),
\]
as we saw above. Hence $|V| = |F|$.

Regardless of the parity of $\tau$, periodicity of the map at time $\tau$ certainly implies that $U^{2\tau} = I$, since there is also periodicity at time $2\tau$.
\end{proof}

An example for when periodicity of the map at time $\tau$ does not imply $U^\tau = I$ is any map with one vertex and more than one face. For such a map, $N = \ones_{\cA}$, meaning that $U \Nhat = \Nhat$, 
%\vincentsays{because 1 is eigenvector for U}
so the map is periodic at time $\tau = 1$. But since $|V| < |F|$, the dimension of the $(-1)$-eigenspace of $U$ is non-zero, so $U \neq I$. Nevertheless, this example leads to the following corollary for maps with a single vertex or face.

\begin{corollary}
\label{cor:1vertex_U^2=I}
For any map with a single vertex or a single face, $U^2 = I$.
\end{corollary}
\begin{proof}
If the map has a single vertex, then the map is periodic at time $1$. If it has a single face, the dual map is periodic at time $1$. In either case, $U^2 = I$ by Theorem \ref{thm:periodic_implies_id}.
\end{proof}

For example, any tree embedded in the plane has a single face, so its transition matrix will satisfy $U^2 = I$. It is natural to ask if perfect state transfer can occur at time $1$ in trees; Proposition \ref{prop:PST_time1} will imply that $K_2$ is the only tree that admits this.

Corollary \ref{cor:1vertex_U^2=I} gives a source for generating examples of maps with periodicity at time $\tau =1$, where $U^\tau \neq  I$. We give an example of this, namely the duals of dipoles with a single face, in Section \ref{sec:1face2vxs}. 
We can ask if such maps also exist for $\tau > 1$. Necessarily, for these maps, $\tau$ must be odd and $|V| \neq |F|$ by Theorem \ref{thm:periodic_implies_id}. More specifically, the lemma below implies that such maps must satisfy $|V| < |F|$:

\begin{lemma}\label{lem:periodic_odd_time} Let $X$ be a map.
\begin{enumerate}[(i)]
    \item Let $\tau > 0$ be odd. The map $X$ is periodic at time $\tau$ if and only if
\[
\col\left(U^{\frac{\tau + 1}{2}}\Nhat\right) \subseteq \col(\Mhat).
\]
In particular, $|V| \leq |F|$.
\item A map $X$ is periodic at time $1$ if and only if $|V| = 1$.
\end{enumerate}
\end{lemma}

\begin{proof}
For part (i), if $X$ is periodic at time $\tau$, then $U^{2\tau} = I$ by Theorem \ref{thm:periodic_implies_id}, so we have that $U^{\tau}$ is an involution and thus $U^{\tau} = (U^T)^{\tau}$. 
Thus, the periodicity implies that $(U^T)^{\tau}\Nhat = \Nhat$ and we obtain that
\[
(U^T)^{\frac{\tau-1}{2}}\Nhat = U^{\frac{\tau+1}{2}}\Nhat
\]
by multiplying by $U^{\frac{\tau+1}{2}}$ on both sides. 
Since $\Nhat = (2Q-I)\Nhat$ and $(2P-I) U^t (2Q-I) = (U^T)^{t-1}$, we can write
\[
(2P-I)U^{\frac{\tau+1}{2}}\Nhat  =(U^T)^{\frac{\tau-1}{2}}\Nhat = U^{\frac{\tau+1}{2}}\Nhat
\]
Thus $U^{\frac{\tau + 1}{2}} \Nhat$ is invariant under $2P - I$, so the columns of $U^{\frac{\tau + 1}{2}} \Nhat$ are in $\col(M)$. In particular, since multiplying by  $U^{\frac{\tau + 1}{2}}$ preserves the orthonormality of the columns of $\Nhat$:
\[
|V| = \dim(\col(\Nhat)) \leq \dim(\col(\Mhat)) = |F|.
\]
For part (ii), consider if $U\Nhat = \Nhat$. Then $\col(\Nhat) \subseteq \col(\Mhat)$ by Lemma \ref{lem:periodic_odd_time}, which implies that $\col(\Nhat) = \vecspan\{\ones\}$ by Theorem \ref{thm:1_eigenspaces}. Hence $|V| = 1$. The other direction was is discussed above.
\end{proof}

Note that this result also implies that periodicity at time $1$ can only occur for maps with a single vertex: if $\col(\Nhat) \subseteq \col(\Mhat)$, then $\col(\Nhat) = \vecspan\{\ones\}$.

In Section \ref{sec:grids}, we give for any $\tau > 0$ a map with $|V| = \tau = |F|$ that satisfies $U^\tau = I$. We then show that, given such a map, we can add a few edges in a way that retains the periodicity at time $\tau$. The newly obtained map has the same number of vertices, but the number of faces has increased. Hence by Theorem \ref{thm:periodic_implies_id}, if $\tau$ is odd, $U^\tau \neq I$. In this way, we show in Lemma \ref{lem:variant1-m-grid} that, for all odd $\tau$, there exists a map that is periodic at time $\tau$ such that $U^\tau \neq I$. This shows that the statement of Theorem \ref{thm:periodic_implies_id} is best possible. 
The following relates periodic maps to perfect state transfer. 

\begin{theorem}
\label{thm:transitive_pst}
Assume that $X$ is a map such that the automorphism group of $X$ acts transitively on the set of vertices $V$. Let $u$ and $v$ be vertices of $X$. The following are true.
\begin{enumerate}[(i)]
    \item If there $uv$-PST at time $\tau$, then for any vertex $x$ there is a unique vertex $y$ such that there is $xy$-PST at time $\tau$. Moreover, $U^{2\tau} = I$.
    \item If there is periodicity at $u$ at time $\tau$, then the map is periodic at time $\tau$.
    \item If there is $uv$-PST at time $1$, then $V = \{u,v\}$.
\end{enumerate}
\end{theorem}
\begin{proof}
For (i), consider any automorphism $\pi$ of $X$, and write $\pi_\cA$ and $\pi_V$ for the permutation matrices that correspond to the action of $\pi$ on the sets of arcs and vertices respectively. By Proposition \ref{prop:regmaps_perms}, we know that $\pi_\cA \Nhat = \Nhat \pi_V$ and $\pi_\cA \Mhat = \Mhat \pi_F$. From this, we deduce that $\pi_\cA$ commutes with both $P$ and $Q$, and hence also with $U$. This implies that
\[
\pi_V^T B_\tau \pi_V = \Nhat^T \pi_\cA^T U^\tau \pi_\cA \Nhat = \Nhat^T U^\tau \Nhat = B_\tau.
\]
Then if $u$ (resp.\ $v$) is mapped to the vertex $x$ (resp.\ $y$) under $\pi$, we have
\[
\Ze_{x}^T B_\tau \Ze_{y} = \Ze_{u}^T \pi_{V}^T B_\tau \pi_{V}\Ze_{v} = \Ze_{u}^T B_\tau \Ze_{u} = 1,
\]
meaning that there is $xy$-PST at time $\tau$ by Lemma  \ref{lem:pst_B_General} (and $yx$-PST by Theorem \ref{thm:pst_reverse_periodicity}(i)). Since the action of the automorphism group is transitive on $V$, any vertex $w$ is the image of $u$ under the action of some automorphism, so $|V|$ can be partitioned into pairs of vertices that admit perfect state transfer to each other at time $\tau$. Theorem \ref{thm:pst_reverse_periodicity}(ii) then implies that the map is periodic at time $2\tau$, so $U^{2\tau} = I$ by Theorem \ref{thm:periodic_implies_id}.
Property (ii) is similar to (i): if there is periodicity at one vertex, there must be periodicity at every vertex, so the map is periodic.

Finally, property (iii) follows from (i) and Proposition \ref{prop:PST_time1}: if there is $uv$-PST at time $1$, then $U^2 = I$. By definition, $u$ and $v$ have the same vertex-degree, so the Proposition applies.

\end{proof}

If a map $X$ has a partition of vertices into pairs, such that $X$ admits perfect state transfer between every pair, then we can give a crude bound on the time of perfect state transfer, using some basic algebraic number theory. In the following lemma,  $\varphi$ denotes the Euler totient function.
\begin{lemma}
    If $U^s = I$ and $U^t \neq I$ for $0<t<s$, then $2\min\{n,f\} \geq \varphi(s')$ for any divisor $s'$ of $s$ such that there exists an eigenvalue of $U$ which is a primitive $s'$th root of unity. 
\end{lemma}

\begin{proof}
    Let $\psi(t):= \psi(\Chat\Chat^T, t)$ denote the minimal polynomial of $\Chat\Chat^T$. Since $\Chat\Chat^T$ has entries in $\rats$, the roots of $\psi(t)$ lie in some field extension of $\rats$; let $\mathbb{K}$ be the splitting field of $\psi$ over $\rats$. 

    Let $\lambda$ be an eigenvalue of $U$. By Theorem \ref{thm:orth_decomposition}, the minimal polynomial of $\lambda$ over $\mathbb{K}$ has degree at most $2$. Let $\mathbb{L}$ be the splitting field of the minimal polynomial of $\lambda$ over $\mathbb{K}$. Since $U^s = I$, we see that every eigenvalue of $U$ must be a $s$th root of unity. Thus, for some $s'$ dividing $s$, we have that $\lambda$ is a primitive $s'$ root of unity and so $\lambda \in \rats(\zeta)$ where $\zeta = e^{\frac{2\pi i}{s'}}$. Since $U$ is a rational matrix, its characteristic polynomial has rational coefficient and thus the algebraic conjugates of $\lambda$ must occur as eigenvalues of $U$ with equal multiplicity as $
    \lambda$ and so we see that $ \rats(\zeta) \subseteq \mathbb{L}$. 

    Now we consider the indices of these field extension and we see that  $[\mathbb{L}: \mathbb{K}] \leq 2$ and
    \[
    [\mathbb{K}:\rats] \leq \deg(\psi(t)) \leq \min(n,f),
    \]
    since $\Chat\Chat^T$ and $\Chat^T\Chat$ have the same minimal polynomial, up to a factor of $t$, and the degree is upper-bounded by size of the matrix. We also have that 
    \[
  [\mathbb{L}:\rats]=  [\mathbb{L}: \mathbb{K}][\mathbb{K}:\rats] \geq [\rats(\zeta) :\rats] = \varphi(s')
    \]
    where $\varphi$ denotes the Euler totient function. 
\end{proof}

A well-known, elementary lower bound for the Euler totient function of a number $n$ is $\varphi(n) \geq \sqrt{\frac{n}{2}}$. We obtain that $s' \leq 2(2\min\{n,f\})^2 = 8 (\min\{n,f\})^2$. Let $S$ be the set of integer $s'$ such that $U$ has an eigenvalue which is a primitive $s'$ root of unity. Since $s$ is the smallest positive integer such that $U^s = I$, we see that $s$ is the least common multiple of the elements of $S$. If $s$ is prime, then $s \in S$ and we have that 
\[
s \leq 8 (\min\{n,f\})^2.
\]
Otherwise, the distinct elements of $S$ are each upper bounded by $8 (\min\{n,f\})^2$ and thus 
\[
s \leq \prod_{s' \in S'} s'\leq (8 (\min\{n,f\})^2)!
\]
where $S'$ is the set of distinct elements of $S$.

\begin{corollary}\label{cor:easybd}
 If map $X$ has a partition of vertices into pairs, such that $X$ admits perfect state transfer between every pair, then the time $\tau$ where perfect state transfer first occurs is upper bounded as follows:
 \[
\tau \leq \frac{1}{2}(8 (\min\{n,f\})^2)!.
 \]
\end{corollary}

\begin{proof}
    If $X$ admits perfect state transfer at time $\tau$ between every pair of vertices, then $X$ is periodic at time $2\tau$ and, by Theorem \ref{thm:periodic_implies_id}, $U^{2\tau} =I$. By the discussion above, we see that $2\tau \leq (8 (\min\{n,f\})^2)!$ and the result follows.
\end{proof}

\section{Periodic maps with $U^{s}= I$}\label{sec:powerofUisI}

Any rotary map which admits perfect state transfer or  periodicity must have the property that some non-zero power of $U$ is the identity matrix, by Theorem \ref{thm:transitive_pst}. In conjunction with the observations from computing powers of $U$ for the maps in the census of regular and chiral maps, as summarized in Section \ref{sec:computations}, we are motivated to study maps for which $U^s = I$ for some $s >0$. 
In this section, we give necessary and sufficient conditions for $U^s =I$ when $s=1,2$ and we give partial characterizations when $s >2$.

 A \textsl{quasi-tree} is an embedded graph with exactly one face. Dually, \textsl{bouquet} is an embedded graph with exactly one vertex. A \textsl{quasi-tree bouquet} is a bouquet that is also a quasi-tree and forms exactly the characterization of maps for which the transition matrix $U$ satisfies $U = I$. Quasi-tree bouquets have been studied in various works, including \cite{YanJin2022, EllEll2022}. Lemma \ref{lem:U=I} gives another characterization of quasi-tree bouquets. 
 
\begin{lemma}
\label{lem:U=I}
The transition matrix satisfies $U = I$ if and only if the map is a quasi-tree bouquet.
\end{lemma}
\begin{proof}
This follows from Lemma \ref{lem:map_and_dual_periodic}, part (ii) of Corollary \ref{lem:periodic_odd_time}, and duality.
\end{proof}

Now, we turn our attention to the cases for which  $U^{s} = I$ for some $s > 1$. 
The maps for which $U^2 = I$ are characterised in the following lemma. Recall our notation for $P =  \Mhat\Mhat^T$ and $Q = \Nhat\Nhat^T$. Recall also that the vertex-face incidence matrix $C$ is given by $C = N^TM$, where the $(v,f)$-entry of $C$ is equal to the number of times the vertex $v$ appears on the facial walk of the face $f$. Recall also that the normalized vertex-face incidence matrix $\Chat$ is given by
$
\Chat = \Nhat^T\Mhat.
$ 

Finally, $J_{A \times B}$ denotes the all-ones matrix indexed by sets $A$ and $B$.
\begin{lemma}
\label{lem:U^2=I} 
The following are equivalent:
\begin{enumerate}[(i)]
    \item $U^2 = I$;
    \item $PQ = |\cA|^{-1} J_{\cA \times \cA}$;
    \item $C = |\cA|^{-1} D J_{V \times F} \Delta$;
    \item $\Chat\Chat^T = |\cA|^{-1}D^{\frac{1}{2}}J_{V\times V} D^{\frac{1}{2}}$;
    \item every vertex $v$ is traversed $|\cA|^{-1}d(v)d(f)$ times by the facial walk of any face $f$, where $d(v)$ is the degree of $v$ and $d(f)$ is the degree of $f$.
\end{enumerate}
\end{lemma}

\begin{proof}
Note that (iii) and (v) are equivalent because of the combinatorial interpretation of $C$: the $(v,f)$-entry of $C$ is equal to $|\cA|^{-1}d(v)d(f)$. We will now show that (i), (ii), (iii) and (iv) are equivalent.

To show that (i) implies (ii), note that if $U^2 = I$, then $U$ is symmetric. Since we can write
\[
U = 4PQ - 2(P + Q) + I,
\]
this implies that $PQ$ must be symmetric. In particular, since $P$ and $Q$ are orthogonal projections, $PQ$ is itself an orthogonal projection. Since $P,Q$ are symmetric matrices, we see that $PQ = (PQ)^T = QP$, so the image of $PQ$ is $\col(N) \cap \col(M)$, which is equal to $\langle \ones_{\cA} \rangle$ by Theorem \ref{thm:1_eigenspaces}(i). Thus $PQ$ is the orthogonal projection onto $\langle{\ones_{\cA}\rangle}$, that is 
$
PQ = |\cA|^{-1} J_{\cA \times \cA}. $

If (ii) holds, then, as we can write $N = QN$ and $M = PM$, we find that
\[
C = N^TM = N^TQPM = |\cA|^{-1}N^T J_{\cA \times \cA} M = |\cA|^{-1} D J_{V \times F}\Delta,
\]
so we see that (ii) implies (iii).

Now assume that (iii) holds. Then
\[
\Chat = D^{-\frac{1}{2}}C\Delta^{-\frac{1}{2}} = |\cA|^{-1} D^{\frac{1}{2}}J_{V \times F} \Delta^{\frac{1}{2}},
\]
so that
\[
\Chat\Chat^T = |\cA|^{-2}D^{\frac{1}{2}} J_{V \times F} \Delta J_{F \times V} D^{\frac{1}{2}} = |\cA|^{-1} D^{\frac{1}{2}}J_{V\times V} D^{\frac{1}{2}},
\]
where the last equality holds because every entry of $J_{V \times F} \Delta J_{F \times V}$ is equal to the sum of the face-degrees, which is $|\cA|$.

Finally, assume that (iv) is true. Then $\Chat\Chat^T$ is a rank-one matrix that has $1$ as an eigenvalue (with eigenvector $D^{\frac{1}{2}}\ones_{V}$). In particular, $\Chat\Chat^T$ has no eigenvalues besides $0$ and $1$, so by Theorem \ref{thm:orth_decomposition}, the eigenvalues of $U$ are in $\{-1,1\}$. This implies (i).
\end{proof}

The cycle $C_n$ embedded on the sphere is a type $(n,2)$ map with $n$ vertices and $2$ faces. Every vertex is traversed once by the facial walk of either face and we have $U^2 = I$ by part (iv) of the above. In general, we state combinatorial characterisation of when $U^2 = I$ for a map of type $(k,d)$ more simply in the following corollary.  

\begin{corollary}
\label{cor:kd_U^2=I}If $X$ is a type $(k,d)$ map, then $U^2 = I$ if and only if every vertex is traversed by the facial walk of any face $\nicefrac{k}{|V|} = \nicefrac{d}{|F|}$ times.
\end{corollary}
\begin{proof}
This follows directly from part (v) of Lemma \ref{lem:U^2=I}. 
\end{proof}

The case $U^2 = I$ accounts for many examples of periodicity found amongst rotary maps. See Section \ref{sec:computations} for more details of the computations.  
By Theorem \ref{thm:transitive_pst}(iii), if the automorphism group of a map $X$ acts transitively on the vertex set $V$, then $uv$-PST at time $1$ implies $U^2 = I$. Conversely, and more generally, we might ask when there can be $uv$-PST at time $1$ if $U^2 = I$. It turns out that this property only occurs for  a small family of maps:

\begin{proposition}
\label{prop:PST_time1}
Let $X$ be a map that satisfies $U^2 = I$, and let $u$ and $v$ be distinct vertices of equal degree. Then there is $uv$-PST at time $1$ if and only if $V = \{u,v\}$.
\end{proposition}
\begin{proof}
Assume first that $V = \{u,v\}$. Then $X$ is an embedding of a $d$-regular graph, where $d$ is the degree of $u$ and $v$. For such a map, as $U^2 = I$, we find by part (iv) of Lemma \ref{lem:U^2=I} that
\[
\Chat\Chat^T = \frac{d}{|\cA|} J_{V \times V} = \frac{1}{|V|} J_{V \times V}.
\]
In this case, $|V| = 2$, so we obtain
\[
B_1 = 2\Chat\Chat^T -I = J_{V \times V} - I.
\]
Then $B_1(u,v) = 1$, so there is $uv$-PST at time $1$ by Lemma \ref{lem:pst_B_General}. 

For the other implication, note that since $U^2 = I$, we know by part (iv) of Lemma \ref{lem:U^2=I} that
\[
\Chat\Chat^T = \Zx \Zx^T,
\]
where $\Zx = |\cA|^{-\frac{1}{2}} D^{\frac{1}{2}} \ones_V$. In particular, $\Zx_w \neq 0$ for all $w \in V$. If there is $uv$-PST at time $1$, then $B_1(u,v) = 1$ by Lemma \ref{lem:pst_B_General}, so we must have $\Zx_u\Zx_v = \frac{1}{2}$. However, $\Zx$ has norm $1$ with entries between $0$ and $1$, which implies that $\Zx_u\Zx_v = \frac{1}{2}$ if and only if $\Zx_u = \Zx_v = \frac{1}{\sqrt{2}}$ and $\Zx_w = 0$ for $w \neq u,v$. We conclude that $V = \{u,v\}$.
\end{proof}

As mentioned in Section \ref{sec:periodicity}, the result above implies that $K_2$ is the only tree that admits perfect state transfer.  
In the following proposition, we show that for $p > 2$ prime, type $(k,d)$-maps that have incidence multiplicity $\alpha$ can only have the property that $U^p = I$ under a very restricted set of circumstances.

\begin{proposition}\label{prop:typekd-U-p}
Let $X$ be an orientable embedding of type $(k,d)$ with incidence multiplicity $\alpha$ and let $p > 2$ be prime. If $U^p = I$, then exactly one of the following three cases holds:
\begin{enumerate}[(i)]
\item $d = \alpha$ and $U = I$; 
\item $d = 2\alpha$, $|V| = |F| = p$ and $\alpha$ is even; 
\item $d = 3\alpha$, $p = 3$, $|V| = |F| = 9$ and $\alpha$ is divisible by $4$.
\end{enumerate}
\end{proposition}
\begin{proof}
Throughout this proof, we will write $n = |V|$, $\ell = |E|$ and $s = |F|$. If $U^p = I$ for some odd prime $p$, then $-1$ is not an eigenvalue of $U$, hence we know by Corollary \ref{cor:U_-1_not_eig} that $n = s$. This also implies that $k = d$, as $nd = sk = |\cA|$. Moreover, by Euler's formula,
\[
n + s - \ell = 2 - 2g.
\]
Since $n = s$, it must be that $\ell$ is even.

The eigenvalues of $U$ that are unequal to $1$ are primitive $p$-th roots of unity. Since $U$ is a rational matrix, any eigenvalue must occur with the same multiplicity as each of its algebraic conjugates. Thus all eigenvalues of $U$ not equal to $1$ have the same multiplicity, say $m \geq 0$. The multiplicities of the eigenvalues add up to the dimension of the whole space, so
\[
|\cA| = m_1 + (p-1)m = \ell + 2g + (p-1)m,
\]
where $m_1 = \ell + 2g$ is the dimension of the $1$-eigenspace of $U$. Since $|\cA| = 2\ell$, we find (using Euler's formula) that
\begin{equation}
\label{eq:m(p-1)}
m(p-1) = \ell - 2g = n + s - 2 = 2(n - 1).
\end{equation}
Let $\theta_1,\ldots, \theta_d$ be the distinct eigenvalues of $U$, with corresponding multiplicities $m_1,\ldots, m_d$. We may assume that $\theta_1 = 1$, in which case $m_2= \ldots = m_d = m$.
The trace of $U$ equals the sum of its eigenvalues, so we can write
\begin{equation*}
\nonumber \tr(U) = m_1  + \sum_{i= 2}^d m_i  \theta_i 
\nonumber = \ell + 2g + m \sum_{i= 2}^d \theta_i  
= \ell + 2g - m
\end{equation*}
since the set of all (non-trivial) primitive $p$-th roots of unity sum  to $-1$. On the other hand, by Lemma \ref{lem:Utrace}, and the fact that $n = s$ and $k = d$, we can write
\[
\tr(U) = \frac{4\alpha n}{d} -4n + |\cA|
\]
for the trace of $U$. From these two expressions for $\tr(U)$, and Euler's formula, we obtain that
\[
m = (\ell + 2g) - \left( \frac{4\alpha n}{d} -4n +2\ell \right) = 2n + 2 - \frac{4\alpha n}{d}.
\]
We observe from \eqref{eq:m(p-1)} that $m$ divides $n-1$, so $m$ is at most $n-1$ and we may rearrange to obtain 
$3 \leq \nicefrac{4\alpha n}{d} - n$,
and thus $d < 4\alpha$. As $d$ is a multiple of $\alpha$, that leaves three possible values for $d$: $d = \alpha$ , $d = 2\alpha$ or $d = 3\alpha$. If $d = \alpha$, then $k = \alpha$. Every vertex is only incident to one face, and every face is only incident to one vertex. This implies that $n = s = 1$, in which case $U = I$ by Lemma \ref{lem:U=I}. This is case (i). If $d > \alpha$, then the map has more than one vertex and one face, so $U \neq I$ and hence $m > 0 $. For $d = 2\alpha$, by \eqref{eq:m(p-1)} and our expression for $m$:
\[
p - 1 = \frac{2n -2}{2n + 2 - \frac{4\alpha n}{d}} = \frac{2n-2}{2} = n - 1,
\]
so $n = s = p$. As stated above, $\ell$ is even. In this case, we have $\ell = p\alpha$, so $\alpha$ must be even. This is case (ii).

Similarly, if $d = 3\alpha$, we find that
\[
p-1 = \frac{2n - 2}{2n + 2 -\frac{4n}{d}} = \frac{n - 1}{1 + \frac{n}{3}} = \frac{3n - 3}{n + 3} < 3.
\]
Since $p > 2$, it must be that $p = 3$. Solving for $n$ then gives $n = s = 9$. But now $2\ell = nd = 27\alpha$, so for $\ell$ to be even, we need that $\alpha = 0 \mod 4$. This is case (iii).
\end{proof}

We note that this implies that maps of type $(k,d)$ with an odd incidence multiplicity $\alpha$ (in particular circular embeddings), yield $U^p \neq I$ for all primes $p > 2$. Note that this includes all toroidal $(m,n)$-grids with $m,n \geq 2$, which we will study more carefully in Section \ref{sec:grids}. 

In Proposition \ref{prop:typekd-U-p}, the maps that satisfy case (i) are precisely all quasi-tree bouquets. For any $p>2$, an example of case (ii) is the toroidal $(1,p)$-grid, which is discussed in Section \ref{sec:grids}. We do not know of any example of case (iii).

\begin{lemma}
\label{lem:CCT_rational}
Let $X$ be a map for which the matrix $\Chat\Chat^T$ has rational eigenvalues. Assume that $U^\tau = I$ for some $\tau > 1$ and $U^s \neq I$ for all $s < \tau$, then $\tau \in \{2,3,4,6,12\}$.
\end{lemma}
\begin{proof}
Let $\xi \neq 1$ be an eigenvalue of $U$. Then $\xi$ is a primitive $r$-th root of unity for some $r$ that divides $\tau$. If also $\xi \neq -1$ (meaning that $r>2$), then as $\xi$ is an eigenvalue of $U$, it is a root of
\[
p(t) = t^2 - (4\hat{\lambda} -2)t + 1
\]
for some eigenvalue $\hat{\lambda}$ of $\Chat\Chat^T$ by Theorem \ref{thm:orth_decomposition}. By assumption this eigenvalue is rational, so $p(t)$ has rational coefficients. Then $p(t)$ must be the minimal polynomial of $\xi$ over $\rats$. In particular, because $\xi$ is a primitive root of unity, $p(t)$ is the $r$-th cyclotomic polynomial. It has degree $2$, which implies that $r \in \{3,4,6\}$. Because $\tau$ is minimal, it is the least common multiple of some non-empty subset of $\{2,3,4,6\}$, meaning that $\tau \in \{2,3,4,6,12\}$.
\end{proof}

We note that the converse partly holds; if $U^\tau = I$ for some $\tau \in \{2,3,4,6\}$ and $U^s \neq I$ for all $s < \tau$, then $\Chat\Chat^T$ has rational eigenvalues. The cyclotomic polynomial for the $2,3,4,6$th roots of unity have degrees $1$ or $2$; any minimal polynomial for these roots over an extension field of the rationals will divide the cyclotomic polynomials. Since the cyclotomic polynomials are only degree $1$ or $2$, we can say exactly what the eigenvalues of $\Chat\Chat^T$ have to be. Let $\xi$ be an eigenvalue of $U$ which is a root of 
\[
p(t) = t^2 - (4\hat{\lambda} -2)t + 1
\]
for some eigenvalue $\hat{\lambda}$ of $\Chat\Chat^T$. If $\xi$ is a third root of unity, then $p(t)$ must divide $t^2 + t + 1$ and thus $\hat{\lambda} = \frac{1}{4}$. Similarly, if $\xi$ is a fourth root of unity, then $p(t)$ must divide $t^2 +1$ and thus $\hat{\lambda} = \frac{1}{2}$. If $\xi$ is a sixth root of unity, then $p(t)$ must divide $t^2 - t + 1$ and $\hat{\lambda} = \frac{3}{4}$. Thus if every eigenvalue of $U$ is $\pm 1$ or a $3$rd, $4$th or $6$th root of unity, then $\Chat\Chat^T$ has only rational eigenvalues. 

In the census of regular and chiral maps with at most $1000$ edges (up to duality and also up to `mirror-duality' for the chiral maps), we found that many have the property that $CC^T$ has all integer eigenvalues. These maps are type $(k,d)$-maps, so $\Chat\Chat^T = \frac{1}{kd}CC^T$ has all rational eigenvalues. We also computed for each map, whether or not there exists some minimal $ r \in \{1,2,\ldots,500\}$ such that $U^r = I$; only the values $r=2,6,12$ appeared in these computations. To give an idea of how common these properties are for regular and chiral maps, we summarize these computations in Table \ref{tab:inteigs}. 

\begin{table}[htbp]
\centering
\begin{tabular}{cccccc}
\hline
\multirow{2}{*}{edges} & \multirow{2}{*}{maps} & \multirow{2}{*}{$\sigma(CC^T)\subset \ints$} & \multicolumn{3}{c}{periodicity} \\
& & & $U^2=I$ & $U^6=I$ & $U^{12}=I$ \\
\hline
     regular & 19685 & 16892 & 8816 & 1439 & 550 \\
     chiral & 4516 & 1884 & 314 & 105 & 12
\end{tabular}
\caption{Statistics on integer eigenvalues and powers of $U$ equalling the identity for regular and chiral maps on up to $1000$ edges. \label{tab:inteigs}}
\end{table}

\section{Infinite families of examples}\label{sec:infexamples}

In this section, we give several infinite families of maps which exhibit perfect state transfer and periodicity. In Section \ref{sec:1face2vxs}, we give an infinite family of maps with two vertices which admit perfect state transfer at time $1$. In Section \ref{sec:grids}, we give two infinite families of grids which admit perfect state transfer and also a variant family which admits periodicity at every vertex at some time $s$, but where $U^{s} \neq I$. The last family shows, in some sense, that the statement of Theorem \ref{thm:periodic_implies_id} is best possible.

\subsection{Dipoles with one or two faces}\label{sec:1face2vxs}

A \textsl{dipole} is a graph with two vertices and no loops. Dipoles encode information in many ways and have been studied in various contexts, see \cite{EllEll2022}. Here we study a  specific rotation for dipoles, which results in  embedding with either one face or two face, depending on the parity of the number of edges. Perfect state transfer occurs at time one in these embeddings of dipoles; they are maps whose automorphism group acts transitively on the vertex set and thus illustrate case (iii) of Theorem \ref{thm:transitive_pst}. 

For $n \in \ints_{> 0}$, consider the map $X_n$ with two vertices $u$ and $v$, with edges $e_{1},\ldots,e_{n}$, and rotation system
\[
    u : (e_1,e_2,\ldots,e_{n}), \qquad
    v : (e_1,e_2,\ldots,e_{n}).
\]
That is, every edge is incident to both vertices and the edges appear in the same order around each vertex. Figure \ref{fig:2vx_2torus} depicts $X_5$ and $X_6$. Let the arcs of the map be given by $a_i$ and $b_i$ for $i=1,\ldots,n$, such that each edge $e_i$ is incident to the pair of arcs $(a_{i},b_{i})$, with $a_i$ and $b_i$ having $u$ and $v$ as their respective tails. In other words, we can write
\[
    u : (a_{1},a_{2},\ldots,a_{n}), \qquad
    v : (b_{1},b_{2},\ldots,b_{n})
\]
for the rotation system with respect to the arcs of the map. If $n$ is odd, the map has a single facial walk given by sequence of arcs
\[
(a_n,b_{n-1},a_{n-2},b_{n-3},\ldots,a_1,b_n,a_{n-1},b_{n-2},a_{n-3},\ldots,b_1)
\]

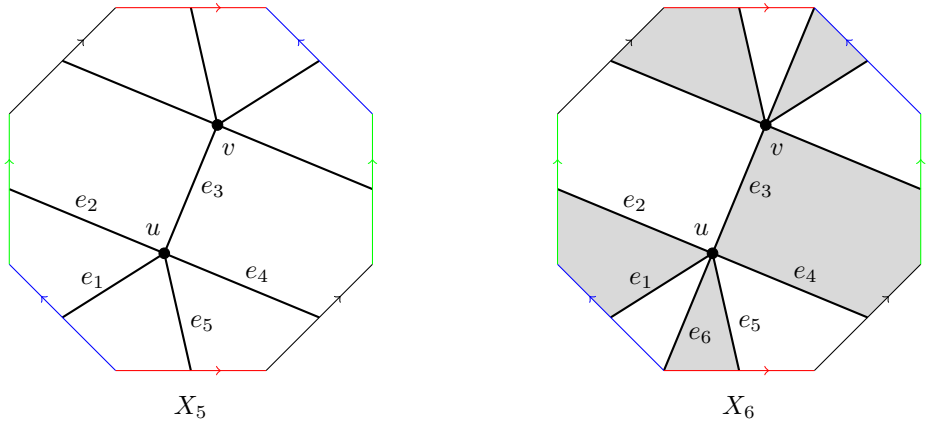
\begin{figure}[htbp]
    \centering
    \begin{tikzpicture}
    	\begin{pgfonlayer}{nodelayer}
    		\node (0) at (-1, 2.4142) {};
    		\node (1) at (1, 2.4142) {};
    		\node (2) at (-2.4142, 1) {};
    		\node (3) at (2.4142, 1) {};
    		\node (4) at (-2.4142, -1) {};
    		\node (5) at (2.4142, -1) {};
    		\node (6) at (-1, -2.4142) {};
    		\node (7) at (1, -2.4142) {};
    		\node [label={[shift={(-.15,0)}]above:$u$}] (8) at (-0.3536, -0.8536) {};
    		\node [label={[shift={(.15,0)}]below:$v$}] (9) at (0.3536, 0.8536) {};
    		\node (10) at (0, 2.4142) {};
    		\node (11) at (0, -2.4142) {};
    		\node (12) at (-2.4142, 0) {};
    		\node (13) at (2.4142, 0) {};
    		\node (14) at (-1.7071, 1.7071) {};
    		\node (15) at (1.7071, 1.7071) {};
    		\node (16) at (-1.7071, -1.7071) {};
    		\node (17) at (1.7071, -1.7071) {};
    		\node [label=below:$X_5$] (18) at (0, -2.5) {};
    	\end{pgfonlayer}
    	\begin{pgfonlayer}{edgelayer}
    	    \draw [thick] (8.center) to node [pos=.7, above] {$e_1$} (16.center);
    		\draw [thick] (8.center) to node [pos=.5, above] {$e_2$} (12.center);
    		\draw [thick] (9.center) to node [pos=.5, right] {$e_3$} (8.center);
    		\draw [thick] (8.center) to node [pos=.6, above] {$e_4$} (17.center);
    		\draw [thick] (8.center) to node [pos=.6, right] {$e_5$} (11.center);
            \draw [thick] (9.center) to node [pos=.7, below] {} (15.center);
    		\draw [thick] (9.center) to node [pos=.6, above] {} (13.center);
    		\draw [thick] (9.center) to node [pos=.6, above] {} (14.center);
    		\draw [thick] (9.center) to node [pos=.6, left] {} (10.center);
    		\draw [thin, style=identify1, color=black] (7.center) to (5.center);
    		\draw [thin, style=identify1, color=black] (2.center) to (0.center);
    		\draw [thin, style=identify1, color=blue] (6.center) to (4.center);
    		\draw [thin, style=identify1, color=blue] (3.center) to (1.center);
    		\draw [thin, style=identify1, color=red] (0.center) to (1.center);
    		\draw [thin, style=identify1, color=red] (6.center) to (7.center);
    		\draw [thin, style=identify1, color=green] (4.center) to (2.center);
    		\draw [thin, style=identify1, color=green] (5.center) to (3.center);
    	\end{pgfonlayer}
    	\filldraw [black]
    	    (8) circle (2pt)
    	    (9) circle (2pt);
    \end{tikzpicture}
    \qquad\qquad\qquad
    \begin{tikzpicture}
    	\begin{pgfonlayer}{nodelayer}
    		\node (0) at (-1, 2.4142) {};
    		\node (1) at (1, 2.4142) {};
    		\node (2) at (-2.4142, 1) {};
    		\node (3) at (2.4142, 1) {};
    		\node (4) at (-2.4142, -1) {};
    		\node (5) at (2.4142, -1) {};
    		\node (6) at (-1, -2.4142) {};
    		\node (7) at (1, -2.4142) {};
    		\node [label={[shift={(-.15,0)}]above:$u$}] (8) at (-0.3536, -0.8536) {};
    		\node [label={[shift={(.15,0)}]below:$v$}] (9) at (0.3536, 0.8536) {};
    		\node (10) at (0, 2.4142) {};
    		\node (11) at (0, -2.4142) {};
    		\node (12) at (-2.4142, 0) {};
    		\node (13) at (2.4142, 0) {};
    		\node (14) at (-1.7071, 1.7071) {};
    		\node (15) at (1.7071, 1.7071) {};
    		\node (16) at (-1.7071, -1.7071) {};
    		\node (17) at (1.7071, -1.7071) {};
    		\node [label=below:$X_6$] (18) at (0, -2.5) {};
    	\end{pgfonlayer}
    	\begin{pgfonlayer}{edgelayer}
    	    \draw [thick] (8.center) to node [pos=.7, above] {$e_1$} (16.center);
    		\draw [thick] (8.center) to node [pos=.5, above] {$e_2$} (12.center);
    		\draw [thick] (9.center) to node [pos=.5, right] {$e_3$} (8.center);
    		\draw [thick] (8.center) to node [pos=.6, above] {$e_4$} (17.center);
    		\draw [thick] (8.center) to node [pos=.6, right] {$e_5$} (11.center);
    		\draw [thick] (8.center) to node [pos=.7, right] {$e_6$} (6.center);
            \draw [thick] (9.center) to node [pos=.7, below] {} (15.center);
    		\draw [thick] (9.center) to node [pos=.6, above] {} (13.center);
    		\draw [thick] (9.center) to node [pos=.6, above] {} (14.center);
    		\draw [thick] (9.center) to node [pos=.6, left] {} (10.center);
    		\draw [thick] (9.center) to node [pos=.5, right] {} (1.center);
    		\draw [thin, style=identify1, color=black] (7.center) to (5.center);
    		\draw [thin, style=identify1, color=black] (2.center) to (0.center);
    		\draw [thin, style=identify1, color=blue] (6.center) to (4.center);
    		\draw [thin, style=identify1, color=blue] (3.center) to (1.center);
    		\draw [thin, style=identify1, color=red] (0.center) to (1.center);
    		\draw [thin, style=identify1, color=red] (6.center) to (7.center);
    		\draw [thin, style=identify1, color=green] (4.center) to (2.center);
    		\draw [thin, style=identify1, color=green] (5.center) to (3.center);
    	\end{pgfonlayer}
    	\filldraw [black]
    	    (8) circle (2pt)
    	    (9) circle (2pt);
    	\begin{pgfonlayer}{background}
    	    \fill[fill=lightgray!60] (8.center) to (6.center) to (11.center) to (8.center);
    	    \fill[fill=lightgray!60] (8.center) to (9.center) to (13.center) to (5.center) to (17.center) to (8.center);
    	    \fill[fill=lightgray!60] (9.center) to (1.center) to (15.center) to (9.center);
    	    \fill[fill=lightgray!60] (8.center) to (12.center) to (4.center) to (16.center) to (8.center);
    	    \fill[fill=lightgray!60] (9.center) to (10.center) to (0.center) to (14.center) to (9.center);
    	\end{pgfonlayer}
    \end{tikzpicture}
    \caption{The maps $X_5$ and $X_6$ on the double torus. We can obtain $X_6$ from $X_5$ by adding an edge $e_6$ that passes through the identified corners of the octagon.\label{fig:2vx_2torus}}
\end{figure}

If $n$ is even on the other hand, the map has two facial walks given by the sequences
\[
(a_n,b_{n-1},a_{n-2},b_{n-3},\ldots,b_1) \qquad \text{and} \qquad (b_n,a_{n-1},b_{n-2},a_{n-3},\ldots,a_1).
\]
Thus the genus of the map is
\[
g = \begin{cases}
\frac{n-1}{2}, & \text{if $n$ is odd;} \\
\frac{n-2}{2}, & \text{if $n$ is even.}
\end{cases}
\]

\begin{theorem}\label{thm:example-pst-time1-2vx}
There is $uv$-PST occurring at time $1$ in $X_n$, for all $n\geq 2$.
\end{theorem}

\begin{proof} We can compute that vertex-face incidence matrix $C$ satisfies
\[
C = \begin{bmatrix} n \\ n \end{bmatrix} \quad \text{ and } \quad C = \frac{1}{2}\begin{bmatrix} n & n \\ n & n\end{bmatrix},
\]
when $n$ is odd and even, respectively. In either case, we obtain that
 $\Nhat^T U \Nhat = J_2 - I_2$ and the result follows.
 \end{proof}

\subsection{Toroidal grids}
\label{sec:grids}

In \cite{PatRagRun2005,Falk2013,AmbPorNah2015}, quantum search is studied  on the toriodal grid. Zhan generalized the unitary operator without the query matrix tothe vertex-face walk, in the two reflections model and details the connection with these search algorithms in \cite[Section 8]{Zha2020}. In this section, we find perfect state transfer and periodicity in some toroidal grids. 

\begin{figure}[htb]
\centering
\begin{subfigure}{0.45\textwidth}
    \centering
    \includegraphics[width=\textwidth]{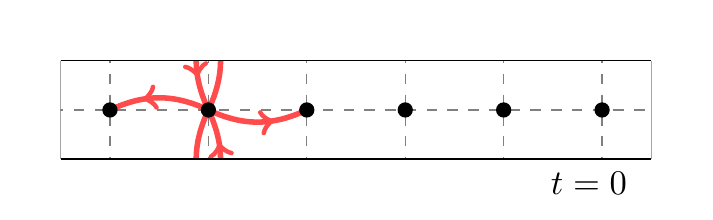}
\end{subfigure}
\begin{subfigure}{0.45\textwidth}
    \centering
    \includegraphics[width=\textwidth]{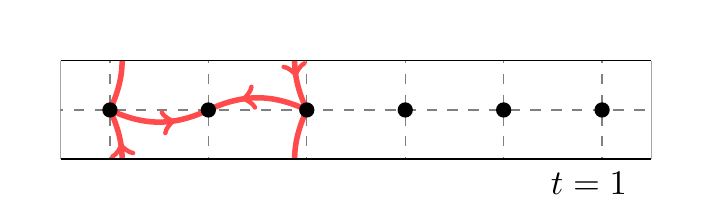}
\end{subfigure}

\vspace{-15pt}

\begin{subfigure}{0.45\textwidth}
    \centering
    \includegraphics[width=\textwidth]{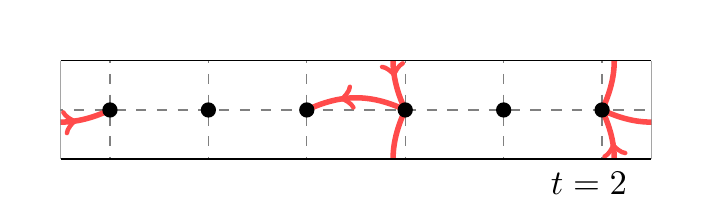}
\end{subfigure}
\begin{subfigure}{0.45\textwidth}
    \centering
    \includegraphics[width=\textwidth]{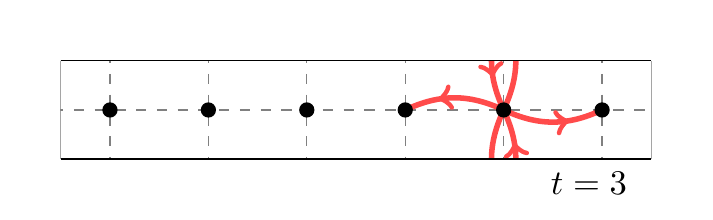}
\end{subfigure}

\vspace{-15pt}

\begin{subfigure}{0.45\textwidth}
    \centering
    \includegraphics[width=\textwidth]{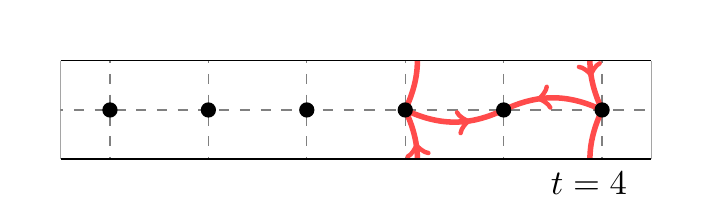}
\end{subfigure}
\begin{subfigure}{0.45\textwidth}
    \centering
    \includegraphics[width=\textwidth]{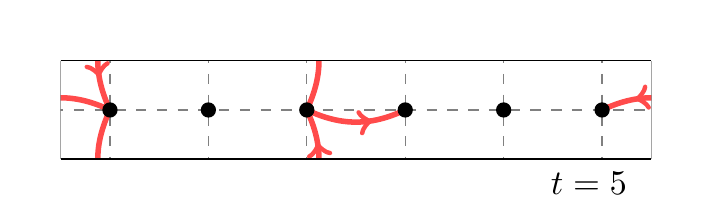}
\end{subfigure}

\vspace{-15pt}

\begin{subfigure}{0.45\textwidth}
    \centering
    \includegraphics[width=\textwidth]{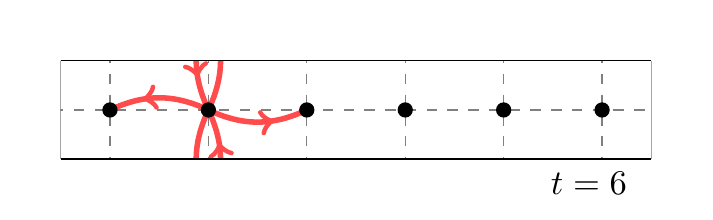}
\end{subfigure}
\begin{subfigure}{0.45\textwidth}
    \centering
    \includegraphics[width=\textwidth]{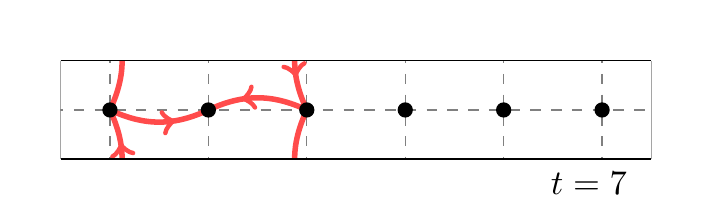}
\end{subfigure}

\caption{Evolution of the walk on the $(1,6)$-grid, with perfect state transfer at time $3$ and periodicity at time $6$. The graphs are drawn on the cut-open torus, where the opposite sides are identified; for visual simplicity, we have omitted the labels on the boundary of the torus. At every step, the arcs with a non-zero amplitude are shown.\label{fig:1_6grid}}
\end{figure}

Figure \ref{fig:1_6grid} shows the evolution of the walk on the $1\times 6$ grid embedded on the torus, starting at the uniform superposition of a vertex. As is suggested by the figure, perfect state transfer occurs at time $t=3$. We will now proceed by giving a rigourous definition of the toroidal grids before proving our main theorems establishing the perfect state transfers in the toroidal $(1,m)$-grids for $m$ even and in the $(2,m)$-grids for $m$ odd. We also alter the $(1,m)$-grids to give an infinite family of examples which admit periodicity at every vertex at time time $s$ but where $U^s \neq 0$. 

Now we proceed by giving a rigourous definition of the toroidal grids, with their rotational systems. The \textsl{toroidal $(n,m)$-grid} has vertex set $V = \ints_n \times \ints_m$ and edge set 
\[
E = \{v_R,v_D \mid v \in V\},
\]
such that for all $(a,b) \in V$:
\begin{itemize}
    \item $(a,b)_R$ is incident with $(a,b+1)$ and $(a,b)$;
    \item $(a,b)_D$ is incident with $(a+1,b)$ and $(a,b)$.
\end{itemize}

For the rotation system, the edges incident with the vertex $(a,b) \in V$ are ordered as follows:
\[
(a,b)_R,\, (a,b)_D,\, (a,b-1)_R,\, (a-1,B)_D.
\]
The corresponding map has genus 1. For example, in Figure \ref{fig:2_3grid}, the top left vertex is vertex $(0,0)$ and  is incident to $(0,0)_R$ (the edge to its right), to $(0,0)_D$ (the edge downwards from $(0,0$), as well as $(1,0)_D$ and $(0,2)_R$. 

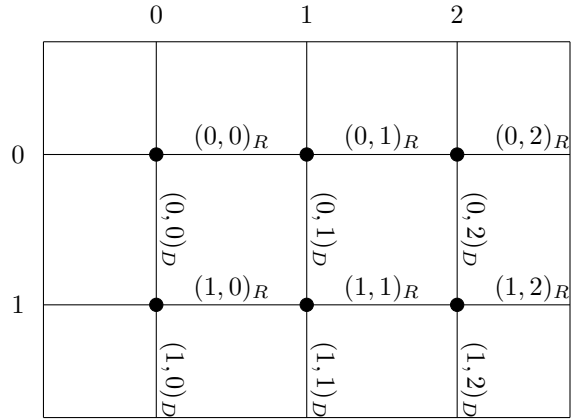
\begin{figure}[hb]
    \centering
    \begin{tikzpicture}
    	\begin{pgfonlayer}{nodelayer}
        	\node [label={[shift={(0,.225)}]center:$(0,0)_R$}] (0r) at (-1, 1) {};
    		\node [label={[shift={(0,.225)}]center:$(1,0)_R$}] (1r) at (-1, -1) {};
    		\node [label={[shift={(0,.225)}]center:$(0,1)_R$}] (2r) at (1, 1) {};
    		\node [label={[shift={(0,.225)}]center:$(1,1)_R$}] (3r) at (1, -1) {};
    		\node [label={[shift={(0,.225)}]center:$(1,2)_R$}] (4r) at (3, -1) {};
    		\node [label={[shift={(0,.225)}]center:$(0,2)_R$}] (5r) at (3, 1) {};
    		%[label={[shift={(0.15,-.1)}, rotate=-90]right:$(0,0)_D$}]
    		\node [label={[shift={(0.225,0)}, rotate=-90]center:$(0,0)_D$}] (0d) at (-2, 0) {};
    		\node [label={[shift={(0.225,0)}, rotate=-90]center:$(1,0)_D$}] (1d) at (-2, -2) {};
    		\node [label={[shift={(0.225,0)}, rotate=-90]center:$(0,1)_D$}] (2d) at (0, 0) {};
    		\node [label={[shift={(0.225,0)}, rotate=-90]center:$(1,1)_D$}] (3d) at (0, -2) {};
    		\node [label={[shift={(0.225,0)}, rotate=-90]center:$(1,2)_D$}] (4d) at (2, -2) {};
    		\node [label={[shift={(0.225,0)}, rotate=-90]center:$(0,2)_D$}] (5d) at (2, 0) {};
    		\node (0) at (-2, 1) {};
    		\node (1) at (-2, -1) {};
    		\node (2) at (0, 1) {};
    		\node (3) at (0, -1) {};
    		\node (4) at (2, -1) {};
    		\node (5) at (2, 1) {};
    		\node (6) at (-3.5, 2.5) {};
    		\node (7) at (-3.5, -2.5) {};
    		\node (8) at (3.5, -2.5) {};
    		\node (9) at (3.5, 2.5) {};
    		\node [label=above:{$0$}] (10) at (-2, 2.5) {};
    		\node [label=above:{$1$}] (11) at (0, 2.5) {};
    		\node [label=above:{$2$}] (12) at (2, 2.5) {};
    		\node [label=left:{$0$}] (13) at (-3.5, 1) {};
    		\node [label=left:{$1$}] (14) at (-3.5, -1) {};
    		\node (15) at (3.5, 1) {};
    		\node (16) at (3.5, -1) {};
    		\node (17) at (-2, -2.5) {};
    		\node (18) at (0, -2.5) {};
    		\node (19) at (2, -2.5) {};
    	\end{pgfonlayer}
    	\begin{pgfonlayer}{edgelayer}
    		\draw (6.center) to (9.center);
    		\draw (7.center) to (8.center);
    		\draw (6.center) to (7.center);
    		\draw (9.center) to (8.center);
    		\draw (10.center) to (17.center);
    		\draw (11.center) to (18.center);
    		\draw (12.center) to (19.center);
    		\draw (13.center) to (15.center);
    		\draw (14.center) to (16.center);
    	\end{pgfonlayer}
    \filldraw [black]
        (0) circle (2.5pt)
        (1) circle (2.5pt)
        (2) circle (2.5pt)
        (3) circle (2.5pt)
        (4) circle (2.5pt)
    	(5) circle (2.5pt);
    \end{tikzpicture}

    \caption{The embedding of toroidal $(2,3)$-grid. The vertices are given by the row and column numbers.  \label{fig:2_3grid}}
\end{figure}

Note that for $n,m \geq 3$, the toroidal $(n,m)$-grid is an embedding of the graph $C_n \Box C_m$. For $n$ or $m$ less than $3$, the underlying graph has multi-edges and/or loops. It is easy to see that the toroidal $(n,m)$-grid is isomorphic to the $(m,n)$-grid as maps; to avoid redundancy, we will state our results only for $n\leq m$. 
In the following lemma, we give an expression for the vertex-face incidence matrix of the toroidal $(n,m)$-grid. Throughout this section, we denote by $P_k$ the cyclic $k \times k$ permutation matrix that maps the standard basis vector $\Ze_{i} \in \cx^{k}$ to $\Ze_{i - 1}$, with the index modulo $k$.

\begin{lemma}
\label{lem:toroidal_incidence}
The vertex-face incidence matrix $C$ of the toroidal $(n,m)$-grid can, with an appropriate ordering of the vertices and faces, be written as
\[
C = (P_n + I_n) \otimes (P_m + I_m).
\]
\end{lemma}
\begin{proof}
First, assume that $n,m \geq 2$. Partition the vertices and faces of $X$ by `row': that is, define
\[
V_i = \{(i,j) : j \in \ints_m\} \quad \text{and} \quad F_i = \{f_{i,j} : j \in \ints_m\}
\]
Every vertex in $V_i$ is incident to two elements of $F_i$, namely $f_{i,j}$ and $f_{i,j-1}$. The submatrix of $C$ corresponding to the vertices of $V_i$ and the faces of $F_i$ is (with the correct ordering) given by $P_m + I_m$. The submatrix of $C$ corresponding to $V_i$ and $F_{i-1}$ is also given by $P_m + I_m$. Consequently, it is not difficult to see that
\begin{equation}
\label{eq:grids_Cmatrix}
C = (P_n + I_n) \otimes (P_m + I_m).
\end{equation}
If instead $n = 1$ and $m > 1$, then there is only one `row' of vertices and faces; we find
\[
C = 2(P_m + I_m).
\]
Similarly, $C = 2(P_n + I_n)$ if $n > 1$ and $m = 1$, and finally $C = \begin{bmatrix}4\end{bmatrix} $ if $n = m = 1$. These expressions coincide with with \eqref{eq:grids_Cmatrix} (note that $P_1 = \begin{bmatrix} 1 \end{bmatrix}$), so that in fact \eqref{eq:grids_Cmatrix} holds for all $n,m \geq 1$.
\end{proof}

In the following lemmas, we look the structure of the matrices $B_t$ for  the toroidal $(1,m)$- and $(2,m)$-grids, in order to prove perfect state transfer occurs. Note that  $P_m^{-t}=(P_m^t)^T$.
\begin{lemma}
\label{lem:B_1m_grid}
For the toroidal $(1,m)$-grid, we have
\begin{equation}
\label{eq:B_1m_grid}
B_t = \frac{1}{2}(P_m^t + P_m^{-t})
\end{equation}
for all $t \in \ints_{\geq 0}$.
\end{lemma}
\begin{proof}
The equation clearly holds for $t = 0$. For $t=1$, note that by Lemma \ref{lem:toroidal_incidence}:
\[
\Chat = \frac{1}{4} C = \frac{1}{2}(P_m + I_m),
\]
since the map has type $(4,4)$. We proceed by induction and see that
\[
B_1 = 2\Chat \Chat^T - I_m 
= \frac{1}{2}(P_m + I_m)(P_m + I_m)^T - I_m 
= \frac{1}{2}(P_m + P_m^{-1}).
\]
By Theorem \ref{thm:B_Chebyshev}, $B_t$ satisfies the recurrence of the Chebyshev polynomials of the first kind,  and thus 
\[
\begin{split}
B_{t+1} &= 2B_{t}B_1 - B_{t-1} \\
&= \frac{1}{2}(P_m^t + P_m^{-t})(P_m + P_m^{-1}) - \frac{1}{2}(P_m^{t-1} + P_m^{-(t-1)}) \\
&= \frac{1}{2}(P_m^{t+1} + P_m^{t-1} + P_m^{-t+1} + P_m^{-t-1}) - \frac{1}{2}(P_m^{t-1} + P_m^{-(t-1)}) \\
&= \frac{1}{2}(P_m^{t+1} + P_m^{-(t+1)}).  
\end{split}
\]
The result now follows.
\end{proof}

Another grid that admits perfect state transfer is the $(2,m)$-toroidal grid, for any odd $m$, as we prove in the following lemma.

\begin{lemma}
\label{lem:B_2m_grid}
For the toroidal $(2,m)$-grid, we have
\begin{equation}
\label{eq:B_2m_grid}
B_t = \frac{1}{4} J_2 \otimes (P_m^t + P_m^{-t} - 2(-1)^{t}I_m) + (-1)^t I_m \otimes I_m
\end{equation}
for all $t \in \ints_{\geq 0}$. 
\end{lemma}
\begin{proof}
We proceed in a similar manner as for the previous lemma. The equation clearly holds for $t = 0$. For $t = 1$, note that 
\[
\Chat = \frac{1}{4}C = \frac{1}{4} J_2 \otimes (P_m + I_m)
\]
by Lemma \ref{lem:toroidal_incidence}. We proceed by induction. Since $J_2J_2^T = 2J_2$, we find that
\[
B_{1} = 2\Chat \Chat^T - I_{2m} 
= \frac{1}{4} J_2 \otimes (P_m + P_m^{-1} + 2 I_m) - I_{2m},
\]
which coincides with \eqref{eq:B_2m_grid}. Let 
$A_t =  P_m^t + P_m^{-t} - 2(-1)^{t}I_m$.
Using Theorem \ref{thm:B_Chebyshev} for the inductive step, we obtain
\[
\begin{split}
B_{t+1} &= 2B_{t}B_1 - B_{t-1} \\
&= 2\left(\frac{1}{4} J_2 \otimes A_t + (-1)^t I_{2m} \right) \left(\frac{1}{4}J_2 \otimes A_1 - I_{2m} \right) - \frac{1}{4} J_2 \otimes A_{t-1} - (-1)^{t-1} I_{2m} \\
&= \frac{1}{8} J_2^2 \otimes A_t A_1 - \frac{1}{2}J_2 \otimes A_t + \frac{1}{2}(-1)^t J_2 \otimes A_1 - 2(-1)^t I_{2m} - \frac{1}{4} J_2 \otimes A_{t-1} - (-1)^{t-1} I_{2m} \\
&= \frac{1}{4} J_2 \otimes \left(A_tA_1 - 2A_t + 2(-1)^t A_1 - A_{t-1}\right) + (-1)^{t+1} I_{2m}.
\end{split}
\]
Since
\[
\begin{split}
A_tA_1 &= (P_m^t + P_m^{-t} - 2(-1)^{t}I_m)(P_m + P_m^{-1} + 2I_m) \\
&= P_m^{t+1} + P_m^{t-1} + P_m^{-t+1} + P_m^{-t-1}
 - 2(-1)^t(P_m + P_m^{-1} ) + 2(P_m^t + P_m^{-t}) - 4(-1)^t I_m \\
&= A_{t + 1} + A_{t-1} - 4(-1)^t I_m - 2(-1)^t( A_1  - 2I_m) + 2(A_t + 2(-1)^t I_m) - 4(-1)^t I_m \\
&= A_{t+1} + A_{t-1} - 2(-1)^t A_1 + 2A_t,
\end{split}
\]
we see that 
$B_{t+1} = \frac{1}{4}J_2 \otimes A_{t+1} + (-1)^{t+1} I_{2m}$,
as desired.
\end{proof}

We now use the form of $B_t$ given by Lemmas \ref{lem:B_1m_grid} and \ref{lem:B_2m_grid}, to obtain our perfect state transfer results on the grids. 

\begin{theorem}\label{thm:gridsPST}
\phantom{stuff}
\begin{enumerate}[(a)]
    \item For the toroidal $(1,m)$-grid,  the map is periodic at time $m$. If $m = 2\ell$ is even, there is perfect state transfer from vertex $(0,i)$ to vertex $(0,i+\ell)$ at time $\ell$ for all $i \in \ints_{m}$.
    \item For the toroidal $(2,m)$-grid, if $m$ is even, the map is periodic at time $m$. If $m$ is odd, there is perfect state transfer from vertex $(0,i)$ to vertex $(1,i)$ at time $m$ for all $i \in \ints_m$, and the map is periodic at time $2m$.
\end{enumerate}
\end{theorem}

\begin{proof}
For part (a), 
 by Lemma \ref{eq:B_1m_grid} and using that $P_m^m = I$, we have
\[
B_m = \frac{1}{2}(P_m^t + P_m^{-t}) = I_m.
\]
This is equivalent to the map being periodic at time $m$. If $m = 2\ell$ is even, then
\[
P_m^\ell = P_m^{-\ell} = \begin{bmatrix}
0 & I_\ell \\ I_\ell & 0,
\end{bmatrix}
\]
which implies that there is perfect state transfer from vertex $(0,i)$ to vertex $(0,i+ \ell)$ at time $\ell$.

For part (b), by Lemma \ref{lem:B_2m_grid}, we have
\[
B_m = \frac{1}{4} J_2 \otimes (P_m^m + P_m^{-m} - 2(-1)^{m}I_m) + (-1)^m I_m \otimes I_m
\]
Clearly, if $m$ is even, then 
$P_m^{m} + P_m^{m} - (-1)^m I_m = 0$,
so $B_m = I_m \otimes I_m$ and the map is periodic at time $m$. If $m$ is odd, then
\[
B_m = J_2 \otimes I_m - I_{m} \otimes I_m.
\]
This matrix swaps the vertices $(0,i)$ and $(1,i)$ for all $i \in \ints_m$, hence there is prefect state transfer between these pairs of vertices. Finally, $B_{2m} = I_m \otimes I_m$, so the map is periodic at time $2m$.
\end{proof}

\begin{figure}[ht]
\centering
\begin{subfigure}{0.26\textwidth}
    \centering
    \includegraphics[width=\textwidth]{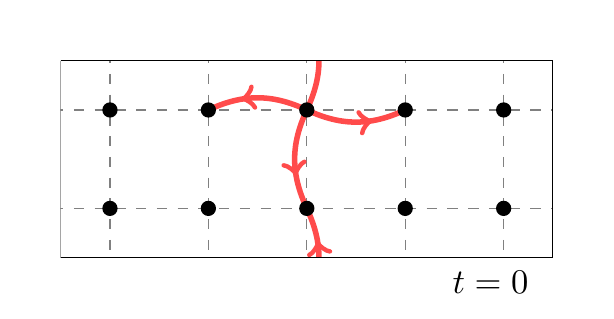}
\end{subfigure}
\kern-1em
\begin{subfigure}{0.26\textwidth}
    \centering
    \includegraphics[width=\textwidth]{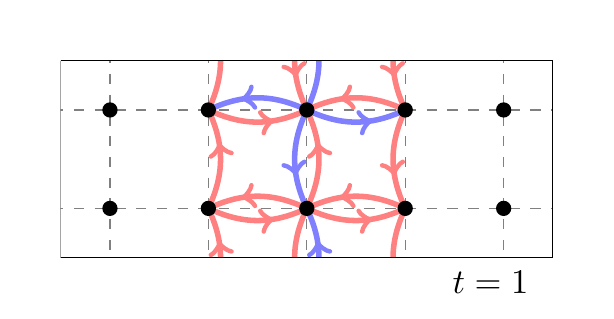}
\end{subfigure}
\kern-1em
\begin{subfigure}{0.26\textwidth}
    \centering
    \includegraphics[width=\textwidth]{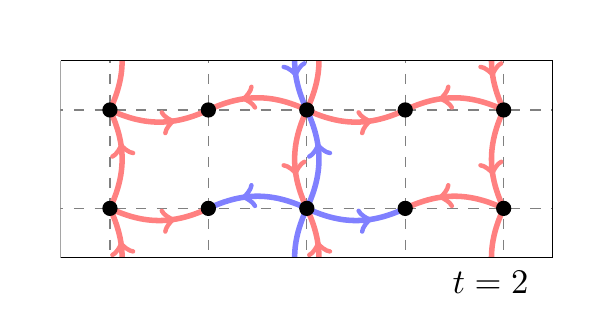}
\end{subfigure}
\kern-1em
\begin{subfigure}{0.26\textwidth}
    \centering
    \includegraphics[width=\textwidth]{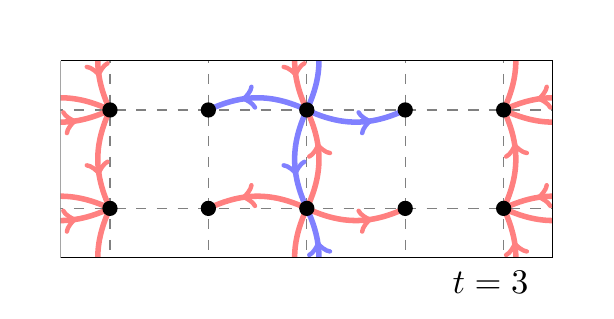}
\end{subfigure}

\vspace{-5pt}

\begin{subfigure}{0.26\textwidth}
    \centering
    \includegraphics[width=\textwidth]{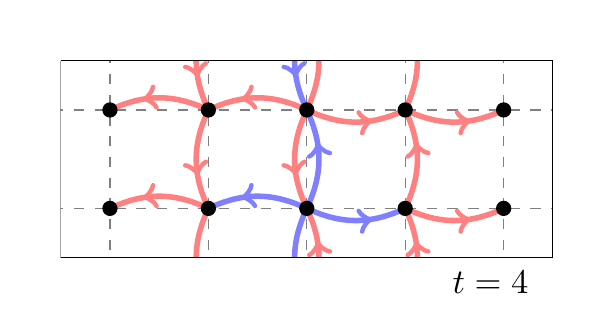}
\end{subfigure}
\kern-1em
\begin{subfigure}{0.26\textwidth}
    \centering
    \includegraphics[width=\textwidth]{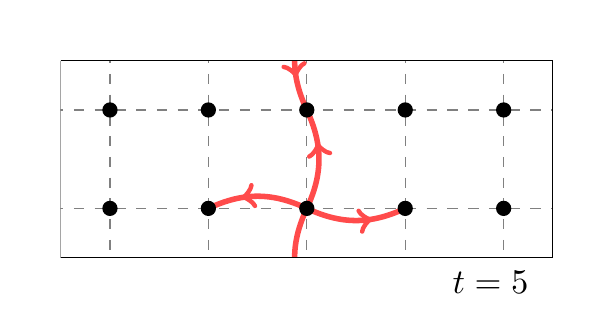}
\end{subfigure}
\kern-1em
\begin{subfigure}{0.26\textwidth}
    \centering
    \includegraphics[width=\textwidth]{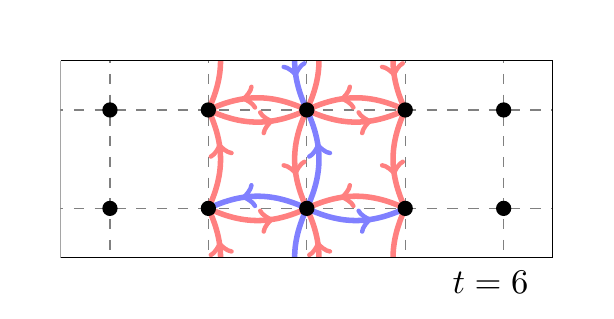}
\end{subfigure}
\kern-1em
\begin{subfigure}{0.26\textwidth}
    \centering
    \includegraphics[width=\textwidth]{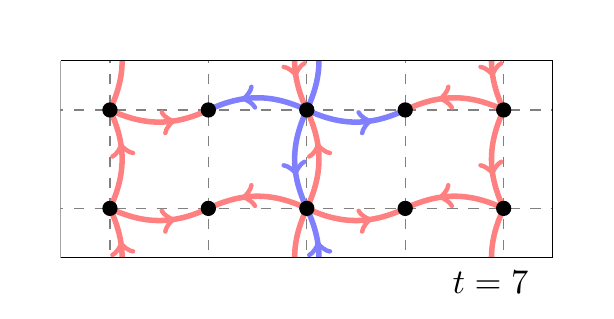}
\end{subfigure}

\vspace{-5pt}

\begin{subfigure}{0.26\textwidth}
    \centering
    \includegraphics[width=\textwidth]{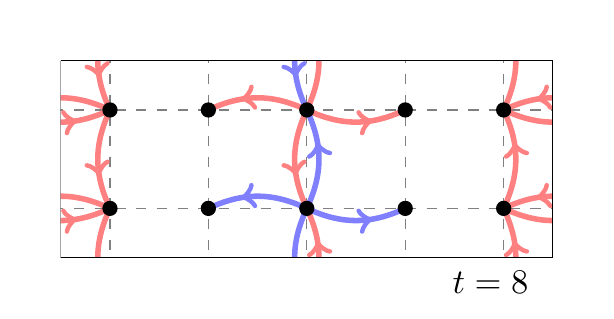}
\end{subfigure}
\kern-1em
\begin{subfigure}{0.26\textwidth}
    \centering
    \includegraphics[width=\textwidth]{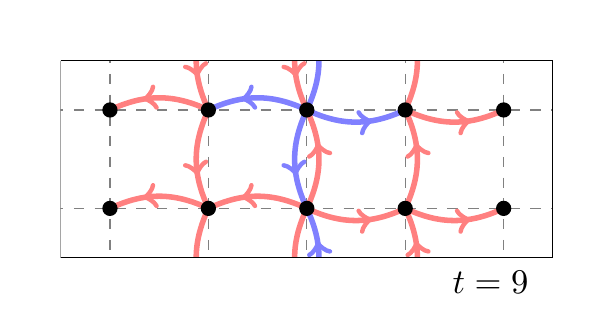}
\end{subfigure}
\kern-1em
\begin{subfigure}{0.26\textwidth}
    \centering
    \includegraphics[width=\textwidth]{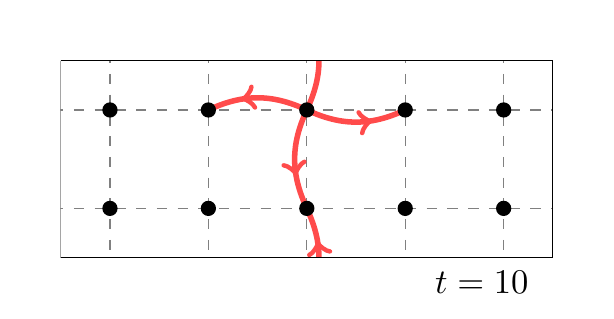}
\end{subfigure}
\kern-1em
\begin{subfigure}{0.26\textwidth}
    \centering
    \includegraphics[width=\textwidth]{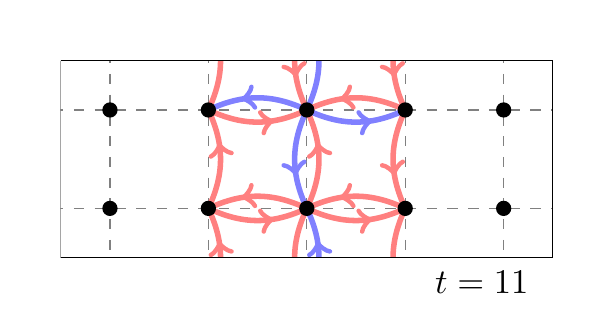}
\end{subfigure}
\caption{The toroidal $(2,5)$-grid, with perfect state transfer at time $5$ and periodicity at time $10$. The graphs are drawn on the cut-open torus, where the opposite sides are identified; for visual simplicity, we have omitted the labels on the boundary of the torus. The color (red, blue) represents the sign of the amplitude of an arc (positive, negative, resp.)\label{fig:2_5grid}}
\end{figure}

Next, we alter the toroidal $(1,m)$-grid to construct an infinite family of maps which have periodicity at every vertex with period $m$, where $U^m \neq I $ when $m$ is odd. Figure \ref{fig:1_5grid_doubled} shows the altered map for the toroidal $(1,5)$-grid.

\begin{lemma}\label{lem:variant1-m-grid}
Let $Y_m$ be the toroidal map obtained from the toroidal $(1,m)$-grid by replacing every non-loop edge by a digon. Then $Y_m$ is periodic with period $m$. If $m$ is odd, then $U^m \neq I$.
\end{lemma}
\begin{proof}
By Lemma \ref{lem:toroidal_incidence}, the vertex-face incidence matrix for the toroidal $(1,m)$-grid can be written as
\[
C = 2(I_m + P_m),
\]
Replacing every non-loop edge of this grid by a digon introduces $m$ new faces of degree $2$, each of which is incident to two vertices. Hence the vertex-face incidence matrix of $Y_m$ can be written as
\[
C_* = \begin{bmatrix}C & \frac{1}{2}C\end{bmatrix}.
\]
The vertices of $Y_m$ have degree $6$, the original faces have degree $4$, and the newly introduced faces degree. The normalized vertex-face incidence matrix of $Y_m$ is hence
\[
\begin{split}
\widehat{C}_* &= \frac{1}{\sqrt{6}} C_* \begin{bmatrix}
\frac{1}{2} I_m & 0 \\ 0 & \frac{1}{\sqrt{2}} I_m
\end{bmatrix} \\
&= \frac{1}{\sqrt{6}} \begin{bmatrix}\frac{1}{2} C & \frac{1}{2\sqrt{2}}C\end{bmatrix}.
\end{split}
\]
This implies that
\[
\widehat{C}_*\widehat{C}_*^T = \frac{1}{6}\left(\frac{1}{4}CC^T + \frac{1}{8}CC^T\right) = \frac{1}{16} CC^T = \Chat\Chat^T.
\]
So $Y_m$ has the same $B_t$-matrix as the toroidal $(1,m)$-grid, for all $t$. Since the latter is periodic at time $m$ by Lemma \ref{lem:B_1m_grid}, the former is as well. Now $|V(X_m)| < |F(Y_m)|$, so for $m$ odd, $U^m \neq I$ by Theorem \ref{thm:periodic_implies_id}.
\end{proof}

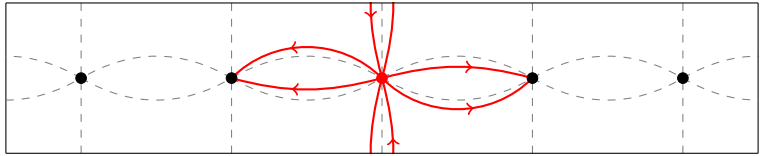
\begin{figure}[htb]
\centering
\begin{tikzpicture}
	\begin{pgfonlayer}{nodelayer}
		\node (0) at (-4, 0) {};
		\node (1) at (-2, 0) {};
		\node (2) at (0, 0) {};
		\node (3) at (2, 0) {};
		\node (4) at (4, 0) {};
		\node (5) at (-6, 1) {};
		\node (7) at (-6, -1) {};
		\node (9) at (-4, 1) {};
		\node (10) at (-2, 1) {};
		\node (11) at (0, 1) {};
		\node (12) at (2, 1) {};
		\node (13) at (3, 1) {};
		\node (14) at (-4, -1) {};
		\node (15) at (-2, -1) {};
		\node (16) at (0, -1) {};
		\node (17) at (2, -1) {};
		\node (18) at (3, -1) {};
		\node (19) at (-2, 2) {};
		\node (20) at (-2, -2) {};
		\node (27) at (-6, 0) {};
		\node (28) at (-7, 1) {};
		\node (29) at (-7, -1) {};
		\node (30) at (-8,0) {};
	\end{pgfonlayer}
	\begin{pgfonlayer}{edgelayer}
	    \clip (-7,-1) rectangle (3,1);
		\draw [gray, dashed, bend left] (0.center) to (1.center);
		\draw [gray, dashed, bend left] (1.center) to (2.center);
		\draw [gray, dashed, bend left] (2.center) to (3.center);
		\draw [gray, dashed, bend left] (3.center) to (4.center);
		\draw [gray, dashed] (5.center) to (7.center);
		\draw [gray, dashed] (9.center) to (14.center);
		\draw [gray, dashed] (10.center) to (15.center);
		\draw [gray, dashed] (11.center) to (16.center);
		\draw [gray, dashed] (12.center) to (17.center);
		\draw [thin] (13.center) to (18.center);
		\draw [thin] (28.center) to (29.center);
		\draw [gray, dashed, bend right] (0.center) to (27.center);
		\draw [gray, dashed, bend right] (27.center) to (0.center);
		\draw [gray, dashed, bend right] (0.center) to (1.center);
		\draw [gray, dashed, bend right] (1.center) to (2.center);
		\draw [gray, dashed, bend right] (2.center) to (3.center);
		\draw [gray, dashed, bend right] (3.center) to (4.center);
		\draw [gray, dashed, bend right] (27.center) to (30.center);
		\draw [gray, dashed, bend right] (30.center) to (27.center);
		\draw [thick, bend left=15, style=directed, color=red] (1.center) to (2.center);
		\draw [thick, bend right=45, style=directed, color=red] (1.center) to (2.center);
		\draw [thick, bend right=15, style=directed, color=red] (1.center) to (20.center);
		\draw [thick, bend right=15, style=directed, color=red] (1.center) to (19.center);
		\draw [thick, bend right=45, style=directed, color=red] (1.center) to (0.center);
		\draw [thick, bend left=15, style=directed, color=red] (1.center) to (0.center);
		\draw [thick, bend right=15, style=directed, color=red] (20.center) to (1.center);
		\draw [thick, bend right=15, style=directed, color=red] (19.center) to (1.center);
		\draw [thin] (28.center) to (13.center);
		\draw [thin] (29.center) to (18.center);
		\filldraw [black]
		    (0) circle (2pt)
		    (2) circle (2pt)
		    (3) circle (2pt)
		    (27) circle (2pt);
        \filldraw [red]
            (1) circle (2pt);
	\end{pgfonlayer}
\end{tikzpicture}
\caption{The toroidal $(1,5)$-grid with doubled non-loop edges, $Y_5$,  has periodicity at time $5$ at every vertex, but $U^5 \neq I$. The graph is drawn on the cut-open torus, where the opposite sides are identified; for visual simplicity, we have omitted the labels on the boundary of the torus.\label{fig:1_5grid_doubled}}
\end{figure}

\section{Computations}\label{sec:computations}

In this section, we offer context and motivation for some of our results. Since vertex-face walks are not yet well-studied in the literature, we performed numerical experiments on a large set of rotary maps to obtain intuition for their behavior. For this, we used a census of all rotary maps having at most $1000$ edges, provided by Conder \cite{Con2009, Con2012, ConDob2001}. In this list, each map is given as a presentation of its automorphism group. We used SageMath \cite{sagemath} to compute  the incidence matrices $N$ and $M$ for each map, and then NumPy \cite{numpy} to compute the transition matrix $U$ and analyse the vertex-face walk on that map. In doing so, we observed some noteworthy behaviour.
For instance, a large proportion of rotary maps up to $1000$ edges, the transition matrix $U$ satisfies the property $U^2 = I$. This provides motivation for our characterizations of such maps in Lemma \ref{lem:U^2=I} and Corollary \ref{cor:kd_U^2=I}.

We also observed that for census of rotary maps, periodicity  only occurred only with a period \linebreak$t \in \{1,2,6,12\}$. Periodic maps with period $t = 1$ have only one vertex (part (ii) of Corollary \ref{lem:periodic_odd_time}). In this case, $U = I$ only if $|F| = 1$. Note that Conder's census omits some degenerate maps, for example the map with one face and one vertex on the torus, as shown in Figure \ref{fig:2loop_torus}.

By Theorem \ref{thm:periodic_implies_id}, maps with a period of $t =2,6$ or $12$ satisfy $U^t = I$. It is interesting to note that for all of the maps that we tested that satisfy $U^t =I$ for $t=6,12$, the transition matrix $U$ does not have any primitive 6th or 12th roots of unity. Moreover, all of the periodic maps are included in a large subset of maps for which the matrix $CC^T$ has integer eigenvalues, which motivated Lemma \ref{lem:CCT_rational}.

Tables \ref{tab:regular_maps} and \ref{tab:chiral_maps} summarize our computations on regular and chiral maps, respectively. Each table shows the number of maps, of each class, that have integer eigenvalues, and in the last three columns for each $t \in \{2,6,12\}$ the number of maps for which $t$ is the smallest time at which $U^t = I$. In these tables, the regular maps are considered up to duality, and the chiral maps up to both duality and mirror image.

\begin{table}[htbp]
\centering
\begin{tabular}{cccccc}
\hline
\multirow{2}{*}{edges} & \multirow{2}{*}{maps} & \multirow{2}{*}{$\sigma(CC^T)\subset \ints$} & \multicolumn{3}{c}{periodicity} \\
& & & $U^2=I$ & $U^6=I$ & $U^{12}=I$ \\
\hline
2-100    & $660$  & $642$  & $482$  & $35$  & $7$   \\ 
101-200  & $1177$ & $1100$ & $696$  & $88$  & $29$  \\ 
201-300  & $1469$ & $1328$ & $778$  & $116$ & $48$  \\ 
301-400  & $1899$ & $1660$ & $875$  & $143$ & $57$  \\ 
401-500  & $1614$ & $1483$ & $862$  & $98$  & $17$  \\ 
501-600  & $2644$ & $2113$ & $997$  & $233$ & $116$ \\ 
601-700  & $1981$ & $1731$ & $955$  & $136$ & $52$  \\ 
701-800  & $3524$ & $2721$ & $1087$ & $266$ & $132$ \\ 
801-900  & $2325$ & $2048$ & $1054$ & $176$ & $60$  \\ 
901-1000 & $2392$ & $2066$ & $1030$ & $148$ & $32$  \\ 
\end{tabular}
\caption{This table shows the number of regular maps that admit periodicity, broken down into ranges of 100 edges.\label{tab:regular_maps}}
\end{table}
\begin{table}[htbp]
\centering
\begin{tabular}{cccccc}
\hline
\multirow{2}{*}{edges} & \multirow{2}{*}{maps} & \multirow{2}{*}{$\sigma(CC^T)\subset \ints$} & \multicolumn{3}{c}{periodicity} \\
& & & $U^2=I$ & $U^6=I$ & $U^{12}=I$ \\
\hline
2-100    & $61$  & $36$  & $5$  & $0$  & $0$ \\ 
101-200  & $176$ & $89$  & $15$ & $3$  & $0$ \\ 
201-300  & $263$ & $143$ & $26$ & $6$  & $1$ \\ 
301-400  & $368$ & $140$ & $26$ & $9$  & $0$ \\ 
401-500  & $393$ & $190$ & $32$ & $6$  & $0$ \\ 
501-600  & $511$ & $228$ & $37$ & $22$ & $8$ \\ 
601-700  & $593$ & $210$ & $32$ & $13$ & $0$ \\ 
701-800  & $769$ & $275$ & $49$ & $18$ & $2$ \\ 
801-900  & $632$ & $317$ & $46$ & $15$ & $1$ \\ 
901-1000 & $750$ & $256$ & $46$ & $13$ & $0$ \\ 
\end{tabular}
\caption{This table shows the number of chiral maps that admit periodicity, broken down into ranges of 100 edges.\label{tab:chiral_maps}}
\end{table}

The transition matrix of the dual of a map $X$ is given by $U^T$, where $U$ is the transition matrix of $X$; thus $U^{t} = I$ if and only if $(U^T)^t = I$. However the period of the periodicity may differ the state spaces are different. If $t > 0$ is odd and there is periodicity with period $2t$ in $X$, it is possible that there is periodicity with period $t$ in the dual. This did occur for some maps in the census, but only at $t=1$ for maps with a single face and more than one vertex, hence $U^2 = I$ (Corollary \ref{cor:1vertex_U^2=I}). Because of this, the dual maps were omitted from the table, and we count the maps for which $U^t = I$, since it applies to both maps under duality.

We did not find any rotary maps with perfect state transfer at time $t > 1$. There were, however, maps with perfect state transfer at time $t = 1$, and such maps satisfy both $U^2 = I$ and $|V| = 2$ by Theorem \ref{thm:transitive_pst}

Results from this paper were used to simplify our computations; for example, Corollary \ref{cor:kd_U^2=I} was used to quickly determine whether a map satisfies $U^2 = I$, and Corollary \ref{cor:statetransferGeneral} was used to simplify computations regarding perfect state transfer and periodicity.

\section{Further directions and open problems}\label{sec:conclusion}

It appears that perfect state transfer (abbreviated hereafter as PST for brevity) is a rare phenomenon. For every time $t > 0$ there exists at least one map which admits PST at time $t$, namely the toroidal $(1,2t)$-grid as discussed in Section \ref{sec:grids}. For odd $t$ there is also the toroidal $(2,t)$ grid with PST at time $t$. Besides toroidal grids, the only maps admitting PST have PST occurring at time $1$ and are maps with only two vertices. We have also searched all orientable embeddings of cubic graphs up to $12$ vertices but PST did not occur for any of these maps. It would be interesting to see more examples of maps that admit PST. In particular, we do not know of any simple graphs admitting PST, other than the planar embedding of $K_2$. To aid in the quest for PST, the following open problem would be of interest:

\begin{openprob}
Does there exist a constant upper bound on a time of PST, in a map admitting PST? 
\end{openprob}

In Corollary \ref{cor:easybd}, we give an upper bound on the time of PST in the special case when every vertex has PST with some other vertex at some time $\tau$. 
In the case of the continuous-time quantum walk, \cite{Kay2010} gives an upper bound on the time of perfect state transfer, if it occurs. It is natural to ask if analytic methods can also be applied in the discrete case for an analoguous result. 

We found PST in the toroidal grids with $n =1,2$, but did not find it anywhere else. The symmetry of these maps imply that if there is $uv$-PST for some vertices $u,v$, then the vertex set must partition into pair where there is  PST between every pair. We make the following conjecture. 

\begin{conjecture}
Let $n,m \geq 3$ such that $(n,m) \neq (4,4)$. Then the toroidal $(n,m)$-grid is not periodic at any time $\tau$. Consequently, there is also no perfect state transfer at any time $\tau$.
\end{conjecture}

Since perfect state transfer in vertex-face walks appears to be a rare phenomenon, we can turn our attention to the several other possible methods of state transfer. In the remainder of this section, we will discuss variations on the notion of PST for the vertex-face walk.

\subsection*{PST between vertices of different degrees}

Recall that in Section \ref{sec:pst}, we originally defined PST between vertices $u$ and $v$ at time $\tau > 0$ as
$U^\tau \Nhat\Ze_u = \Zx$
where $\Zx \in \cx^\cA$ is a unit length vector that satisfies $\Nhat\Ze_w \circ \Zx = 0$ for all $w \neq v$ (i.e. $\Zx$ is any superposition of the arcs incident to $v$).  We then restricted the definition of PST to be between vertices $u$ and $v$ of equal degree, in which $\Zx$ would necessarily have to equal $\Nhat\Ze_v$ (Lemma \ref{lem:uv_PST_simplified}). If we allow $u$ and $v$ to have different degrees however, $\Zx$ can be any superposition of the arcs that are incident to $v$. This raises the question: are there maps for which this more general type of $uv$-PST occurs between vertices of different degrees? As was discussed in Section 4, this can only happen if the degree of $v$ is smaller than the degree of $u$, as was discussed in Section \ref{sec:pst}. We do not know of any examples  of perfect state transfer between vertices of different degree in a vertex-face walk.

\subsection*{``Reverse'' PST }

Consider the unique genus $0$ embedding of the path $P_3$ as depicted in Figure \ref{fig:P3}. 

\begin{figure}[hbt]
    \centering
    \begin{tikzpicture}[scale=1]
    	\begin{pgfonlayer}{nodelayer}
    		\node [label=below:$v$] (0) at (0, 0) {};
    		\node [label=below:$u$] (1) at (-3, 0) {};
    		\node [label=below:$w$] (2) at (3, 0) {};
    		\node [label=right:$f$] (3) at (0.5,1) {};
    	\end{pgfonlayer}
    	\begin{pgfonlayer}{edgelayer}
    		\draw [style=directed, thick, color=vcolour, bend right=15] (1.center) to node [pos = 0.6, below] {$a_1$} (0.center);
    		\draw [style=directed, thick, color=vcolour, bend right=15] (2.center) to node [pos = 0.6, above] {$a_4$} (0.center);
    		\draw [style=directed, thick, color=kcolour, bend right=15] (0.center) to node [pos = 0.6, below] {$a_3$} (2.center);
    		\draw [style=directed, thick, color=kcolour, bend right=15] (0.center) to node [pos = 0.6, above] {$a_2$} (1.center);
    		\draw [gray, dashed] (1.center) to (0.center);
    		\draw [gray, dashed] (0.center) to (2.center);
    	\end{pgfonlayer}
    	\filldraw [black]
    	    (0) circle (2pt)
    	    (1) circle (2pt)
    	    (2) circle (2pt);

    \end{tikzpicture}
    \caption{The vertex-face walk on $P_3$, with in blue the initial state $\Nhat \Ze_v$ and in red the state $R\Nhat\Ze_v$, at time $1$.\label{fig:P3}}
\end{figure}
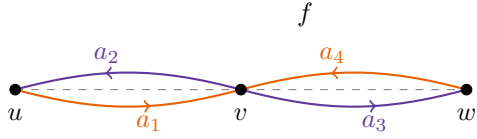

Because $P_3$ is a tree, we have $U^2 = I$. Like for all trees other than $K_2$, there is also no PST at time $1$. However, $U$ sends the uniform superposition of the arcs incident to $v$ to the reverse arcs:
\[
U\Nhat\Ze_v = R\Nhat\Ze_v,
\]
where $R$ is the arc-reversal matrix. The same is true for the central vertex of any star graph $K_{1,n}$. Generally, we can say a map $X$ admits \textsl{reverse $uv$-PST} at time $\tau$ if 
\[
U^\tau \Nhat \Ze_u = R\Nhat\Ze_v
\]
for vertices $u$ and $v$ at some time $\tau > 0$. 
Besides the star graphs, this happens for any embedding of the graph on two vertices with a number of parallel edges between them, such as the digon from Figure \ref{fig:digon&dual}. All of these examples  at time $\tau = 1$. 
A natural problem to ask would be the following. 
\begin{openprob}
What are the classes of orientable maps admit reverse $uv$-PST? Further, are there examples where it occurs for the first time at  time $\tau > 1$? 
\end{openprob}

\begin{figure}[htb]
    \centering
    \begin{tikzpicture}
    	\begin{pgfonlayer}{nodelayer}
    		\node (0) at (-4, 0) {};
    		\node (1) at (-2, 0) {};
    		\node (2) at (0, 0) {};
    		\node (3) at (2, 0) {};
    		\node (4) at (4, 0) {};
    		\node (5) at (-5, 1) {};
    		\node (6) at (5, 1) {};
    		\node (7) at (-5, -1) {};
    		\node (8) at (5, -1) {};
    		\node (9) at (-4, 1) {};
    		\node (10) at (-2, 1) {};
    		\node (11) at (0, 1) {};
    		\node (12) at (2, 1) {};
    		\node (13) at (4, 1) {};
    		\node (14) at (-4, -1) {};
    		\node (15) at (-2, -1) {};
    		\node (16) at (0, -1) {};
    		\node (17) at (2, -1) {};
    		\node (18) at (4, -1) {};
    		\node (19) at (-2, 2) {};
    		\node (20) at (-2, -2) {};
    		\node (21) at (2, 2) {};
    		\node (22) at (4, 2) {};
    		\node (23) at (2, -2) {};
    		\node (24) at (4, -2) {};
    		\node (25) at (-5, 0) {};
    		\node (26) at (5, 0) {};
    	\end{pgfonlayer}
    	\begin{pgfonlayer}{edgelayer}
    	    \clip (-5,-1) rectangle (5,1);
    		\draw [gray, dashed] (9.center) to (14.center);
    		\draw [gray, dashed] (10.center) to (15.center);
    		\draw [gray, dashed] (11.center) to (16.center);
    		\draw [gray, dashed] (12.center) to (17.center);
    		\draw [gray, dashed] (13.center) to (18.center);
    		\draw [gray, dashed] (25.center) to (26.center);
    		\draw [thick, style=directed, bend right=15, color=kcolour] (1.center) to (19.center);
    		\draw [thick, style=directed, bend right=15, color=kcolour] (1.center) to (20.center);
    		\draw [thick, style=directed, bend right=15, color=kcolour] (20.center) to (1.center);
    		\draw [thick, style=directed, bend right=15, color=kcolour] (19.center) to (1.center);
    		\draw [thick, style=directed, bend right=15, color=vcolour] (3.center) to (21.center);
    		\draw [thick, style=directed, bend right=15, color=vcolour] (23.center) to (3.center);
    		\draw [thick, style=directed, bend right=15, color=vcolour] (22.center) to (4.center);
    		\draw [thick, style=directed, bend right=15, color=vcolour] (4.center) to (24.center);
    		\draw [thick, style=directed, bend right=15, color=kcolour] (1.center) to (2.center);
    		\draw [thick, style=directed, bend right=15, color=kcolour] (1.center) to (0.center);
    		\draw [thick, style=directed, bend right=15, color=vcolour] (3.center) to (4.center);
    		\draw [thick, style=directed, bend right=15, color=vcolour] (4.center) to (3.center);
    		\draw [ultra thin] (5.center) to (6.center);
    		\draw [ultra thin] (7.center) to (8.center);
    		\draw [ultra thin] (7.center) to (5.center);
    		\draw [ultra thin] (8.center) to (6.center);
    	\end{pgfonlayer}
        \begin{pgfonlayer}{background}
            \fill[vcolour!30]
    	        (12.center) to (13.center) to (18.center) to (17.center) to (12.center);
        \end{pgfonlayer}
        \filldraw [black]
    	    (0) circle (2pt)
    	    (2) circle (2pt)
    	    (3) circle (2pt)
    	    (4) circle (2pt);
    	\filldraw [kcolour] (1) circle (2pt);
    \end{tikzpicture}

    \caption{Vertex-face PST in the toroidal $(1,5)$-grid. The arcs incident to the purple vertex are sent to the arcs incident the orange face in $3$ steps. The graph is drawn on the cut-open torus, where the opposite sides are identified; for visual simplicity, we have omitted the labels on the boundary of the torus.\label{fig:vxf_PST}}
\end{figure}
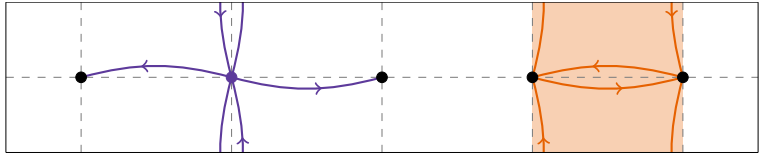

\subsection*{Vertex-face PST} Each step in the vertex-face walk on a map $X$ can be thought of as taking one step in $X$ and then one step in the dual $X^*$, each step corresponding to one of the two reflections that form the transition matrix $U$. It is hence natural to define the notion state transfer between a vertex and a face; we say that a map admits \textsl{vertex-face perfect state transfer} if 
\[
U^\tau \Nhat e_u = \Mhat e_f
\]
at some time $\tau > 0$ for some vertex $u$ and face $f$. 

For example, the toroidal $(1,5)$-grid admits vertex-face perfect state transfer between the `antipodal' vertex-face pairs, as shown in Figure \ref{fig:vxf_PST}. In this example, the map is periodic at time $5$ and Lemma \ref{lem:periodic_odd_time} gives us that $U^{3} \Nhat e_u $ is in the column space of $\Mhat$; in fact, in this case,   $U^{3} \Nhat e_u $  is a column of $\Mhat$. Thus, in a sense, we can view the vertex-face PST as a strengthening of the condition in Lemma  \ref{lem:periodic_odd_time}.

One can ask if it is easier to generate prolific examples of this form of PST. Other basic questions to investigate include the following:
\begin{itemize}
    \item If vertex-face PST occurs at time $\tau$, is there periodicity at time $2\tau$? 
    \item If vertex-face PST occurs between vertex $v$ and face $f$, can it also occur at $v$ and $f' \neq f$ at some other time? 
    \item What are some structural properties that $v$ and $f$ have to satisfy, when vertex-face PST occurs between $v$ and $f$? 
\end{itemize}
For the third question, we are motivated by our example, in which $v,f$ are antipodal, in some sense. 

% \bibliographystyle{plain}
% \bibliography{qwalk}

\end{document}